\documentclass[11pt, a4paper]{article}

\usepackage[T1]{fontenc}
\usepackage[utf8]{inputenc}
\usepackage{mathpazo}
\usepackage{amsmath, amssymb, amsthm}
\usepackage{geometry}
\usepackage{hyperref}
\usepackage{booktabs}
\usepackage{array}
\usepackage{tabularx}
\usepackage{microtype}
\usepackage{enumitem}
\usepackage{xcolor}
\usepackage{titlesec}
\usepackage{graphicx}
\usepackage{tikz}
\usetikzlibrary{arrows.meta, calc, patterns}

\hypersetup{colorlinks=true, linkcolor=black, citecolor=black, urlcolor=blue}

\theoremstyle{plain}
\newtheorem{theorem}{Theorem}[section]
\newtheorem{proposition}[theorem]{Proposition}

\newtheorem{corollary}[theorem]{Corollary}
\newtheorem{conjecture}[theorem]{Conjecture}
\newtheorem*{conjecture*}{Conjecture}
\newtheorem{heuristic}[theorem]{Heuristic Proposal}

\theoremstyle{definition}
\newtheorem{definition}[theorem]{Definition}

\newtheorem{remark}[theorem]{Remark}

\newcommand{\R}{\mathbb{R}}
\newcommand{\GL}{\mathrm{GL}}

\newcommand{\so}{\mathfrak{so}}
\newcommand{\Hol}{\mathrm{Hol}}
\newcommand{\Om}{\Omega}
\newcommand{\Sig}{\mathcal{A}}
\newcommand{\Gam}{\Gamma}
\newcommand{\del}{\partial}
\newcommand{\norm}[1]{\left\|#1\right\|}

\newcommand{\Gr}{\mathcal{G}}   % Green's function (distinct from structure group G)

\titleformat{\section}{\large\bfseries}{\thesection.}{0.5em}{}
\titleformat{\subsection}{\normalsize\bfseries\itshape}{\thesubsection.}{0.5em}{}

\begin{document}
% ================================================================

\title{From Gradient Descent to Potential Theory:\\
A Geometric Dictionary for Binary Classification}

\author{Catalin Vasii\thanks{Ness Digital Engineering.
Email: \texttt{cvasii23@gmail.com}}}

\date{}

\maketitle

\begin{abstract}
We propose a precise \textbf{dictionary} between binary classification in machine
learning and classical potential theory on vector bundles. Classifiers are
parallel sections of vector bundles over the data space; training labels are
Dirichlet boundary conditions on the section; the kernel of an RKHS interpolant
is the Green's function of an elliptic operator on the data manifold; and
backpropagation is the flat-geometry limit of an exact geometric problem
(\S\ref{subsec:backprop}).

The central identification is this. For $L = (\Delta_g + \kappa^2)^\nu$ with
$\nu > n/2$, consider the $L$-harmonic interpolation boundary value problem:
find the minimum-$H^\nu$-norm classifier satisfying $Lf = 0$ away from the data,
subject to $f(x_i) = y_i \cdot r$. This is precisely the problem that RKHS
interpolation and kernel ridge regression already solve. The representer theorem
of Kimeldorf and Wahba \cite{kimeldorf1971some} and the SPDE--Green's function
correspondence of Lindgren, Rue, and Lindstr\"om \cite{lindgren2011spde} supply
the analytic content; what we add is the potential-theoretic reading
(Theorem~\ref{thm:svm-ym}): the kernel is the Green's function of $L$, the
interpolation coefficients are electrostatic capacitances, and the decision
boundary is the zero equipotential of the resulting field. The condition
$\nu > n/2$ is essential and excludes the Dirichlet energy ($\nu = 1$) in
dimension $n \geq 2$, where point data has zero capacity. The hard-margin SVM is
a margin-constrained variant of the same object (Corollary~\ref{cor:svm-ym}).

Two further results. For finite data on any smooth manifold, flat $O(2)$
solutions always exist (Theorem~\ref{thm:flat}); and the density of $O(2)$
$L$-harmonic interpolants among continuous classifiers
(Theorem~\ref{thm:universal}) is the universal kernel theory of Steinwart
\cite{steinwart2001} and Micchelli--Xu--Zhang \cite{micchelli2006} in geometric
form. Over non-contractible bases the geometry becomes load-bearing: the parity
invariant on $S^1$ and the Euler class on $S^2$ are obstructions no choice of
classifier can evade. Elsewhere the connection is determined by the section and
carries no independent information; making it an independent field with its own
energy is the subject of a companion paper.
\end{abstract}



\bigskip

\bigskip

% ================================================================
\section{Introduction}
% ================================================================

The dominant paradigm in machine learning treats classification as an optimisation
problem: minimise a loss functional over a space of parameters. The geometry
underlying this process is typically Riemannian --- gradient descent follows the
metric on parameter space --- but the classifier itself is rarely given a geometric
interpretation.

In this paper we propose a different geometric frame. We observe that in both the
support vector machine (SVM) --- a maximum-margin classifier whose decision function
is defined via a kernel $K$ and support vector weights --- and the one-hidden-layer
neural network without activation or bias --- where the classifier is the linear map
$f(x) = W_2 W_1 x$ --- the classifier arises from a \emph{section} of a vector
bundle over the data space $M$. This is a new \emph{framing}, not a new object:
SVMs and neural networks are well-understood, but their identification as sections
of vector bundles with the training problem as a Dirichlet BVP --- and the resulting
geometric dictionary with Riemannian geometry and potential theory --- is the
contribution. Learning is reformulated as finding a \emph{connection} on that
bundle for which the section is horizontal.

Before stating the principles, we briefly recall the geometric objects involved,
for the reader's convenience.

A \textbf{vector bundle} $E \to M$ is a family of vector spaces (the
\emph{fibers}) parameterised smoothly by points of the base space $M$. In our
setting $M \subseteq \R^n$ is the data space and each fiber $E_x \cong \R^m$ is a
copy of $\R^m$ sitting above the data point $x \in M$. A \textbf{section}
$s: M \to E$ is a smooth choice of one vector $s(x)$ in each fiber --- it assigns
a hidden representation to every data point. The \textbf{structure group}
$G \subseteq \GL(m)$ is the group of linear automorphisms of the fiber regarded
as symmetries of the bundle. Two sections $s$ and $s' = g \cdot s$ related by a
smooth map $g: M \to G$ (a \emph{gauge transformation}, or \emph{$G$-transformation})
represent the same geometric object: the physically meaningful quantity is not $s$
itself but its \emph{gauge equivalence class} (or \emph{$G$-equivalence class})
$[s]$. The structure group therefore controls which transformations of the hidden
representation are invisible to the classifier. A \textbf{connection} $\nabla$ on
$E$ is a rule for differentiating sections: it specifies, at each point of $E$,
a decomposition of the tangent space into a \emph{vertical} part (along the fiber)
and a \emph{horizontal} part (transverse to it). A section is called
\emph{horizontal} (or \emph{parallel}) if its derivative lies entirely in the
horizontal part, meaning it does not move in the fiber direction as $x$ varies
across $M$. The \textbf{curvature} $F$ of the connection measures the failure of
the horizontal distribution to be integrable: zero curvature means parallel
transport around any loop returns the fiber to itself, while non-zero curvature
means transport around a loop can rotate the fiber non-trivially. Finally, the
\textbf{readout map} $\varphi: E \to \R$ is a map that is linear on each fiber
separately: for every $x \in M$, the restriction $\varphi|_{E_x}: E_x \cong
\R^m \to \R$ is a linear functional. Equivalently, $\varphi$ is a global section
of the dual bundle $E^*$. Note that $\varphi$ is \emph{not} required to be linear
as a map on the total space $E$ itself, which is not a vector space. Concretely,
$\varphi$ is the output weight vector $W_2 \in \R^{1\times m}$, acting fiberwise
as $\varphi(s(x)) = W_2 s(x)$. The classifier is $\sigma = \varphi \circ s:
M \to \R$, and the decision boundary is $\Gam = \sigma^{-1}(0)$.

\textbf{Notation convention.} Throughout this paper, $\varphi$ denotes the
readout map $E \to \R$, while $\phi$ (a different glyph) denotes the $O(2)$
fibre angle function $\phi: M \to \R$ that parametrises the section as
$s = r(\cos\phi, \sin\phi)^\top$. The two symbols are visually close but play
entirely different roles; this convention is maintained consistently.

Three principles organise the framework:
\begin{enumerate}[label=(\arabic*), itemsep=2pt]
  \item Two independent geometric choices replace the two arbitrary choices of
        classical ML: the \textbf{structure group} $G$ replaces the activation
        function; the \textbf{Riemannian metric} $g$ replaces the kernel.
        Both are developed precisely in \S\ref{sec:nn}.
  \item A \textbf{section} of the bundle together with a fiberwise linear readout
        $\varphi$ and Dirichlet conditions at the data points constitutes a
        classification problem.
  \item The \textbf{two-stage energy-minimising solution} satisfying those Dirichlet
        conditions is the unique geometrically preferred solution --- found by
        solving a classical PDE, not by gradient descent.
\end{enumerate}

\paragraph{Scope, stated up front.}
In most of this paper the connection is \emph{derived} from the section rather
than being an independent field: horizontality determines it, so choosing a
solution amounts to choosing a smooth angle function. This makes some results
easy --- Theorem~\ref{thm:flat} is a partition-of-unity argument --- and it means
the reader should not look for gauge-theoretic content in
\S\ref{sec:xor}--\S\ref{sec:curvature}. The framework earns its keep in the
topology of non-contractible bases (\S\ref{sec:sphere}), in the energy selection
that produces the Green's function and the capacitance system
(\S\ref{sec:yang-mills}), and in the dictionary itself
(\S\ref{sec:dictionary}). Making the connection an independent variable with its
own energy is the subject of the companion paper \cite{Vasii2026Paper2}.
Remark~\ref{rem:connection-derived} states this precisely at the point where it
matters most.

The paper is organised as follows. \S2 surveys related work. \S3 recalls
background on connections. \S4 develops the SVM case. \S5 develops the neural
network case and the refined definitions. \S6 works out XOR with $G = O(2)$.
\S7 develops the flower-shaped problem. \S8 computes curvature and proves every
finite classification problem admits a flat solution (Theorem~\ref{thm:flat}).
\S9 develops text classification on spheres. \S10 discusses the structure group
hierarchy and proves the geometric universality theorem
(Theorem~\ref{thm:universal}). \S11 connects the structure group to data
symmetry. \S12 proposes persistent homology as a guide to the structure group.
\S13 develops classification as a geometric BVP with exact solutions. \S14
presents a numerical experiment. \S15 crystallises the framework as a precise
dictionary. \S16 states the remaining open conjecture and open problems.

The paper has three arcs. \textbf{Arc 1} (\S3--\S9): examples and the bundle
framework --- SVM, XOR, flower, curvature, spheres. \textbf{Arc 2}
(\S10--\S12): which structure group to use --- hierarchy, universality,
symmetry efficiency, and TDA. \textbf{Arc 3} (\S13--\S14): the exact geometric
solution --- $L$-harmonic interpolation, the BVP, and the numerical experiment.
Readers who prefer a deductive presentation may read \S5 (definitions) then
\S13 (the general BVP) before the examples.

% ================================================================
\section{Related Work}\label{sec:related}
% ================================================================

The literature relating geometry and machine learning divides cleanly into two
programmes, both distinct from the present work.

\paragraph{Geometry used \emph{for} machine learning.}
The geometric deep learning programme \cite{bronstein2021} places neural network
architectures on a group-theoretic foundation: convolutions, graph networks, and
transformers are unified as instances of blueprint invariant or equivariant maps
under a symmetry group. Gauge-equivariant convolutional networks
\cite{cohen2019gauge} extend this to principal bundles, associated bundles, and
gauge transformations, giving architectures whose filters are literally sections of
associated bundles over the input manifold. This work shares vocabulary with the
present paper --- bundles, connections, gauge transformations --- but uses them
differently: it constrains the architecture to be equivariant, while we identify
the classification \emph{problem} as a BVP on a bundle. Miolane et al.'s Geomstats
library \cite{miolane2020geomstats} and the broader programme of geometric
statistics \cite{miolane2025euclid} develop Riemannian and Lie-group geometry as
computational tools for ML, with applications to shape analysis, medical imaging,
and structured data. Information geometry \cite{amari2000} equips the space of
probability distributions with a Riemannian structure (the Fisher metric) and
studies gradient descent in those terms. Topological data analysis
\cite{edelsbrunner2010} extracts homological invariants from data, applied here in
\S\ref{sec:tda} to guide structure group selection. All of these use geometric
tools to improve or understand ML. None identifies ML problems as classical
geometric BVPs.

\paragraph{Neural networks described by kernel theory.}
The neural tangent kernel (NTK) \cite{jacot2018ntk} shows that an infinitely wide
neural network trained by gradient descent is equivalent to kernel gradient descent
with a specific kernel derived from the network architecture. In the infinite-width
limit, training dynamics become linear and the solution is a kernel regression
estimator. This line of work establishes that \emph{wide networks behave like kernel
machines} --- a convergence result. The present paper works in the opposite
direction: we identify the kernel machine itself as the exact solution to a
geometric BVP, and show that the network is an iterative approximation to that
exact problem.

\paragraph{The Matérn--Green's function correspondence.}
The identification of Matérn kernels as Green's functions of elliptic operators is
known in spatial statistics. Lindgren, Rue, and Lindström \cite{lindgren2011spde}
establish an explicit link between Gaussian fields with Matérn covariance and
solutions of SPDEs of the form $(\Delta + \kappa^2)^{\nu/2} x = \mathcal{W}$
(white noise), used computationally to obtain sparse precision matrices via finite
elements. Whittle \cite{whittle1963} established the spectral-domain version of
this correspondence earlier. In the ML kernel literature, this connection is noted
in the Gaussian process regression context \cite{scholkopf2002}. Our Theorem~\ref{thm:svm-ym}
identifies the same correspondence as the geometric content of RKHS interpolation:
the kernel is the Green's function, the interpolation coefficients are capacitances,
and the solution is a harmonic potential. The contribution is the geometric
\emph{interpretation} --- the dictionary between the kernel machine and the
classical BVP --- rather than the analytic correspondence itself.

\paragraph{What is new.}
The present paper does not use geometry as a tool for ML, nor does it show that one
ML object approximates another. It establishes an \emph{isomorphism}: the binary
classification problem, as formulated by the SVM and RKHS interpolation, \emph{is}
a classical Dirichlet BVP on a vector bundle; the kernel \emph{is} the Green's
function; the training \emph{is} potential theory. The arbitrariness of current
ML --- the choice of activation, kernel, learning rate --- resolves into two
independent geometric choices (structure group $G$, Riemannian metric $g$) with
precise mathematical meaning. To our knowledge, no prior work frames binary
classification as an element of the solution space of a geometric BVP in this
way.

% ================================================================
\section{Background: Connections on Vector Bundles}
% ================================================================

\subsection{Line bundles and the abelian case}

Let $M$ be a smooth manifold \cite{kobayashi1963}. A \emph{line bundle} over $M$ is a rank-1 vector
bundle $\pi\colon E \to M$. A connection is specified by a 1-form $A \in \Om^1(M)$,
giving covariant derivative $\nabla_A f = df + A \cdot f$. A section
$f\colon M \to \R$ is \emph{horizontal} if $\nabla_A f = 0$, integrating to
$f(x) = C\exp(-\int_\gamma A)$ on simply connected domains. The curvature is
$F = dA \in \Om^2(M)$.

\subsection{Vector bundles and the non-abelian case}

A rank-$m$ vector bundle $E \to M$ with structure group $G \subseteq \GL(m)$
carries a connection
\[
  \Sig = \Sig_\mu\,dx^\mu \in \Om^1(M,\,\mathfrak{g}),
\]
where $\mathfrak{g} = \mathrm{Lie}(G)$. The covariant derivative of a section
$s\colon M \to \R^m$ is $\nabla_\mu s = \del_\mu s + \Sig_\mu(x)\cdot s$, and
the curvature is
\[
  F = d\Sig + \Sig\wedge\Sig \in \Om^2(M,\,\mathfrak{g}).
\]
The term $\Sig\wedge\Sig$ involves the Lie bracket of $\mathfrak{g}$ and vanishes
when $\mathfrak{g}$ is abelian, but is generically non-zero otherwise.

% ================================================================
\section{Support Vector Machines as a Line Bundle Problem}\label{sec:svm}
% ================================================================

\subsection{The bundle and the connection}

A \emph{support vector machine} (SVM) \cite{scholkopf2002} is a binary classifier that finds a
decision boundary of maximum margin between two classes. Given labeled training
data $\{(x_i, y_i)\}$ with $x_i \in \R^n$ and $y_i \in \{+1,-1\}$, the SVM
with kernel $K: \R^n \times \R^n \to \R$ produces a decision function:
\[
  f(x) = \sum_i \alpha_i\,y_i\,K(x_i, x) + b,
\]
where $K$ is a positive definite \emph{kernel function}, $\alpha_i \geq 0$ are
weights supported on the \emph{support vectors}, and $b \in \R$ is a bias term.
The kernel implicitly maps the data into a (possibly infinite-dimensional)
reproducing kernel Hilbert space \cite{scholkopf2002}.

The line bundle framework applies to any real-valued classifier $f: M \to \R$.
Let $M = \R^n$ and $E = M \times \R$ with structure group $\R^*$. The
classification problem has tuple
\[
  \bigl(M = \R^n,\; g = g_{\mathrm{Eucl}},\; E = M\times\R,\;
        G = \R^*,\; \varphi = \mathrm{id},\; \mathcal{D}\bigr),
\]
where $g_{\mathrm{Eucl}} = \sum_{i=1}^n(dx^i)^2$ is the canonical flat metric on
$\R^n$. The Laplace--Beltrami operator reduces to the standard Laplacian
$\Delta_g = \sum_i\partial_i^2$. Its Green's function --- the Newton/Riesz kernel
$\Gr(x,y) \propto |x-y|^{2-n}$ (for $n \geq 3$) or $-\frac{1}{2\pi}\log|x-y|$
(for $n=2$) --- is singular on the diagonal and not positive definite, hence not
directly usable as an RKHS kernel. The Matérn regularisation of
\S\ref{subsec:kernel} resolves this by replacing $\Delta_g$ with
$(\Delta_g + \kappa^2)^\nu$, giving a positive-definite kernel. The decision
function $f$ defines a section $\sigma(x) = (x, f(x))$. Horizontality
$\nabla_A f = 0$ forces:
\begin{equation}\label{eq:svm-conn}
  A \;=\; -\frac{df}{f} \;=\; -d\log|f|,
\end{equation}
defined on $M\setminus\Gam$ where $\Gam = f^{-1}(0)$ is the decision boundary.

Two distinct estimators arise within this framework, and they must not be
conflated:

\begin{description}[itemsep=4pt]
\item[RKHS interpolation.] Minimise $\|f\|^2_{\mathcal{H}_K}$ subject to
  \emph{equality} constraints $f(x_i) = y_i \cdot r$ at every training point.
  The solution is the $L$-harmonic interpolant $f(x) = \sum_i \alpha_i y_i K(x_i,x)$
  with $\boldsymbol{\alpha} = \mathbf{K}^{-1}\mathbf{v}$ — a single linear solve,
  dense $\boldsymbol{\alpha}$ (Theorem~\ref{thm:svm-ym}).
\item[Hard-margin SVM.] Minimise $\|f\|^2_{\mathcal{H}_K}$ subject to
  \emph{inequality} constraints $y_i f(x_i) \geq 1$ (the margin).
  The solution has the same form but with sparse $\boldsymbol{\alpha}$: only
  support vectors (points where $y_i f(x_i) = 1$) contribute.
  The SVM is the RKHS interpolant on the endogenously determined support vector
  set; the quadratic programme is needed precisely because that set is unknown
  a priori.
\end{description}

The geometric framework developed in \S\ref{sec:yang-mills} characterises the
RKHS interpolant as the minimum-$H^\nu$-norm $L$-harmonic function satisfying the
prescribed Dirichlet data — classical potential theory. The SVM adds the margin
constraint on top of this geometric picture.

\subsection{Flatness and the singular locus}

The connection \eqref{eq:svm-conn} is flat on $M\setminus\Gam$: since
$A = -d\log|f|$ is exact, $F = dA = 0$. On simply connected $M$, all the geometry
is concentrated in the logarithmic singularity of $A$ along $\Gam$.

% ================================================================
\section{Neural Networks: Refined Definitions}\label{sec:nn}
% ================================================================

\subsection{The gauge symmetry}

We consider a \emph{one-hidden-layer neural network} for binary classification on
$\R^n$ \cite{goodfellow2016}, in the simplified setting of \textbf{no activation function and no bias
term}. The omission of activations is deliberate: we will show that the structure
group of the bundle plays precisely the geometric role that activations play in the
standard setting, making activations redundant once the bundle geometry is correctly
specified. The omission of biases is a simplification that does not affect the
geometric structure.

The network computes:
\[
  f(x) \;=\; W_2\,W_1\,x, \qquad
  W_1 \in \mathrm{Mat}(m, n),\quad W_2 \in \mathrm{Mat}(1, m),
\]
where $x \in \R^n$ is the input, $W_1$ maps to a hidden layer of dimension $m$,
and $W_2$ collapses to a scalar output. In the standard network, activations would
be applied to $W_1 x$ before multiplying by $W_2$; here we omit them. The
resulting map is linear in $x$, and we will see that this forces the decision
boundary to be a hyperplane --- precisely the geometric content of Remark~\ref{rem:linear}.

The map $x \mapsto W_2 W_1 x$ is invariant under $G\cdot(W_1,W_2)=(GW_1,W_2G^{-1})$
for any $G\in\GL(m)$. This defines a principal $\GL(m)$-bundle over the space of
effective classifiers, with associated rank-$m$ vector bundle $E = M\times\R^m$
over data space $M = \R^n$. The classification problem has tuple
\[
  \bigl(M = \R^n,\; g = g_{\mathrm{Eucl}},\; E = M\times\R^m,\;
        G = O(m),\; \varphi = (1,0,\ldots,0),\; \mathcal{D}\bigr),
\]
where $g_{\mathrm{Eucl}}$ is the canonical flat metric on $\R^n$.

\subsection{Sections, readouts, and the refined definitions}

\begin{definition}\label{def:cp}
A \emph{classification problem} is a tuple $(M, g, E, G, \varphi, \mathcal{D})$
where:
\begin{itemize}[itemsep=2pt]
  \item $(M, g)$ is the data manifold with a Riemannian metric $g$, which
        determines the Laplace--Beltrami operator $\Delta_g$ and hence the kernel
        $K = \Gr_{\Delta_g}$ (the geometry of the data space);
  \item $E \to M$ is a vector bundle with structure group $G$, which determines
        what connections exist and what decision boundaries are achievable
        (the geometry of the fiber);
  \item $\varphi: E\to\R$ is a fiberwise linear readout (a section of $E^*$);
        concretely $W_2 \in \R^{1\times m}$;
  \item $\mathcal{D} = \{(x_i, y_i)\}$ is a labeled dataset, $y_i\in\{+1,-1\}$.
\end{itemize}
The metric $g$ and the structure group $G$ are \emph{independent} geometric
choices playing distinct roles: $g$ determines the geometry of the base $M$ and corresponds to
the choice of kernel; $G$ acts on the fiber $E_x$ and corresponds to the
expressive capacity of the classifier.
\end{definition}

\begin{definition}\label{def:solution}
Fix a classification problem $(M,g,E,G,\varphi,\mathcal{D})$ with $E$ a
\emph{trivial} bundle, and an operator $L = (\Delta_g+\kappa^2)^\nu$,
$\nu > n/2$. Call a pair $(\Sig, s)$ \emph{admissible} if $s$ is nowhere zero
and parallel for $\Sig$ on $M \setminus \Gam$, and satisfies the Dirichlet
conditions $\varphi(s(x_i)) = y_i \cdot r$. A \emph{solution} is a gauge
equivalence class $[\Sig, s]$ selected from the admissible pairs in two stages:
\begin{enumerate}[label=(\roman*), itemsep=2pt]
  \item minimise the Yang--Mills energy
        $\mathcal{E}_{\mathrm{YM}}[\Sig] = \int_M\norm{F_\Sig}^2\,\mathrm{dvol}_g$;
  \item among the minimisers of (i), minimise the potential energy
        $\mathcal{E}_L[u] = \langle u, Lu\rangle = \norm{u}^2_{H^\nu}$,
        where $u$ is the scalar potential determined by $s$ through $\varphi$
        (for $G=\R^+$, $u = f$; for $G=O(2)$ with $s = r(\cos\phi,\sin\phi)^\top$,
        $u = \phi$).
\end{enumerate}
The \emph{decision boundary} is $\Gam = \{x \in M : \varphi(s(x)) = 0\}$.
\end{definition}

\begin{remark}[Why trivial bundles, and what the two stages do]\label{rem:two-stage}
Two points of care.

\emph{First}, the case split of earlier drafts --- use $\int\norm{F}^2$ when
$F \not\equiv 0$ and the potential energy otherwise --- is inadmissible, since
$F$ is a property of the candidate, not of the problem: such a test makes the
objective depend on the object it is meant to choose. The lexicographic order
above is well defined for every admissible pair. On a trivial bundle stage~(i)
is achieved with value $0$ by the whole flat family
(Theorem~\ref{thm:flat} exhibits one member), and stage~(ii) does the selecting.

\emph{Second}, the restriction to trivial bundles is not cosmetic. A parallel
section satisfies $F_\Sig\, s = 0$; for an $SO(2)$-connection on a rank-2 bundle
the curvature is $F = f\!\cdot\!J$ with $J$ invertible, so $F_\Sig\,s = 0$ forces
$s = 0$ wherever $F \neq 0$. Hence a nowhere-zero parallel section exists only
if the connection is flat, and on a bundle with $e(E) \neq 0$ --- where no flat
connection exists (\S\ref{sec:sphere}) --- the admissible set is empty and
Definition~\ref{def:solution} says nothing. The curved case requires relaxing
``parallel'' to \emph{covariantly harmonic}, $D_A^*D_A\phi = 0$, with $A$ an
independent field; that is the Yang--Mills--Higgs setting of the companion paper
\cite{Vasii2026Paper2} and is not developed here.

\emph{Third}, for $G = \R^+$ the classifier $f$ vanishes on $\Gam$, so
$A = -d\log|f|$ is singular there and $s = f$ is parallel only on
$M\setminus\Gam$ --- as the definition states.
\end{remark}

\begin{remark}[Horizontal vs covariantly harmonic]
Throughout this paper, solutions are described as \emph{horizontal} (parallel)
sections: $\nabla_\Sig s = 0$. This is the flat-connection limit. In the
Yang--Mills--Higgs framework of the companion paper \cite{Vasii2026Paper2},
the section instead satisfies $D_A^* D_A \phi = 0$ away from the data
--- it is \emph{covariantly harmonic}, not horizontal. On flat connections
($F = 0$) and on contractible $M$, the two conditions coincide: $D_A = \nabla$
and covariant harmonicity reduces to ordinary harmonicity, which on contractible
domains is equivalent to the flat horizontal condition. Horizontality
is thus the degenerate limit of covariant harmonicity, appropriate for this
paper's regime.
\end{remark}

\subsection{Two independent geometric choices}

A classification problem $(M, g, E, G, \varphi, \mathcal{D})$ involves two
independent geometric choices that play entirely different roles and should not
be conflated.

\textbf{The structure group $G$} acts on the \emph{fiber} $E_x \cong \R^m$. It
determines: (i) the gauge symmetries; (ii) the Lie algebra $\mathfrak{g}$ in which
the curvature lives; (iii) the \textbf{orbit geometry} --- which $G$-invariant
subsets of the fiber the section is constrained to inhabit.

For $G = O(2)$: since $O(2)$ acts by isometries on $\R^2$, the orbit of any
non-zero vector is the sphere $S^1(r) = \{\|s\| = r\}$. The section is constrained
to this circle, and the angle $\phi$ parametrises its position on $S^1(r)$. Sign
changes in $\sigma = s^1 = r\cos\phi$ occur when $\phi$ passes through $\pi/2$ ---
at that point $s = (0, r) \neq 0$. The section \emph{rotates through the
perpendicular direction} rather than vanishing.

For $G = \R^*$ (line bundle): the section is a scalar function; sign changes require
it to pass through zero. There is no $G$-invariant sphere to rotate on.

This is the precise content of the dictionary entry activation $\leftrightarrow G$:
an activation function (tanh, sigmoid) constrains the hidden representation to an
interval or curve $(-1,1)$ in the output --- a ``sphere'' in 1D. The structure group
$O(m)$ constrains the section to $S^{m-1}(r)$ --- a sphere in $m$ dimensions. Both
are sphere-type constraints on the hidden representation; the structure group
determines the dimension and geometry of that sphere. The specific path traced on
the sphere as $x$ varies is the connection (the angle function $\phi(x)$); the
sphere itself is the orbit of $G$. Crucially, $G$ does not correspond to a kernel
or to a metric --- as a \emph{structure} group it has no counterpart in the data
space $M$.

\begin{remark}[Structure group vs.\ symmetry group]\label{rem:two-roles}
A caveat for \S\ref{sec:symmetry}. A structure group acts on fibres; a symmetry
group of the data acts on the base. These are different roles, and the
symmetry-efficiency conjecture (Conjecture~\ref{conj:symmetry}) needs both at
once: it compares $G$-equivariant sections, which presupposes an action of $G$
on $M$ as well as on the fibre. The conjecture is therefore not a statement
about structure groups alone but about pairs $(G \curvearrowright M,\;
\rho: G \to O(m))$ --- a symmetry group of the base together with a
representation on the fibre, as in equivariant network theory. Where
$G_{\mathrm{data}}$ appears below it is in this doubled sense.
\end{remark}

\textbf{The Riemannian metric $g$} determines the geometry of the \emph{base} $M$.
It determines
the Laplace--Beltrami operator $\Delta_g$, whose Green's function $\Gr_g(x,y)$
is the kernel of the classifier: $K(x,y) = \Gr_g(x,y)$. Different choices of $g$
give different kernels. In the traditional SVM, the kernel is chosen empirically;
here, choosing $g$ is choosing the geometric structure on the data space. The
metric has no counterpart in the fiber.

The dictionary entry is therefore:
\[
  \text{activation function} \;\longleftrightarrow\; G \quad(\text{fiber geometry}),
  \qquad
  \text{kernel } K \;\longleftrightarrow\; g \quad(\text{base geometry}).
\]
Both are geometric choices, but they are independent and govern different parts
of the bundle.

\begin{remark}[Linearity without activation]\label{rem:linear}
When $G = \GL(m)$ and no activation is present, the composed map $x \mapsto
W_2 W_1 x$ is a linear functional on $\R^n$, so the decision boundary is always
a hyperplane. The section $s(x) = W_1 x$ is linear in $x$, and the connection
$\Sig_\mu = -W_1 e_\mu (W_1 x)^\dagger$ has $F = 0$ wherever $s \neq 0$. The
non-abelian term $\Sig\wedge\Sig$ contributes nothing beyond the abelian case:
enlarging the structure group to $\GL(m)$ without introducing non-trivial fiber
transformations leaves the classifier linear.
\end{remark}

% ================================================================
\section{The XOR Problem with $G = O(2)$}\label{sec:xor}
% ================================================================

\subsection{Setup and failure of the abelian case}

The XOR dataset consists of four points in $M = \R^2$:
\[
  (0,0) \mapsto -1,\quad (1,1) \mapsto -1,\quad
  (0,1) \mapsto +1,\quad (1,0) \mapsto +1.
\]
This section studies \textbf{Problem A} (finite labeled data): only the four
point values are prescribed; the global angle function $\phi$ is not unique.
The particular ansatz \eqref{eq:phi-xor} below is one solution; any other smooth
$\phi$ satisfying the same four Dirichlet conditions is equally valid, and may
have different geometric properties (different decision boundary, different
oscillation count). The classification problem has tuple
\[
  \bigl(M = \R^2,\; g = g_{\mathrm{Eucl}},\; E = M\times\R^2,\;
        G = O(2),\; \varphi = (1,0),\; \mathcal{D}_{\mathrm{XOR}}\bigr).
\]
No linear classifier separates them. With $G = \R^*$, any continuous $f$ with the
correct signs requires $\Gam = f^{-1}(0)$ to separate the plane into two regions
containing diagonally opposite pairs. This is topologically impossible with a
single connected curve in $\R^2$ without enclosing a point of the wrong class:
a topological obstruction, not a failure of optimisation.

\begin{figure}[ht]
\centering
\begin{tikzpicture}[scale=2.4]
  \fill[gray!18]
    (-0.5, -0.5+0.707) -- (1.5, 1.5+0.707) --
    (1.5,  1.5-0.707) -- (-0.5, -0.5-0.707) -- cycle;
  \draw[thick, black, dashed]
    (-0.5, -0.5+0.707) -- (1.5, 1.5+0.707);
  \draw[thick, black, dashed]
    (-0.5, -0.5-0.707) -- (1.5, 1.5-0.707);
  \draw[-{Stealth}] (-0.35,0) -- (1.35,0) node[right] {$x^1$};
  \draw[-{Stealth}] (0,-0.35) -- (0,1.35) node[above] {$x^2$};
  \draw (1,0.04)--(1,-0.04) node[below, font=\small] {$1$};
  \draw (0.04,1)--(-0.04,1) node[left,  font=\small] {$1$};
  \node[font=\small\itshape, gray!60!black] at (0.5,  0.5)  {class $-1$};
  \node[font=\small\itshape]               at (-0.18, 0.75) {$+1$};
  \node[font=\small\itshape]               at (1.18,  0.25) {$+1$};
  \node[font=\footnotesize, rotate=45] at (1.05, 1.42)
    {$|x^1\!-\!x^2|=\tfrac{1}{\sqrt{2}}$};
  \draw[thick] (0,0) circle (0.055) node[below left, font=\small] {$(0,0)$};
  \draw[thick] (1,1) circle (0.055) node[above right, font=\small] {$(1,1)$};
  \filldraw[fill=black!80, draw=black]
    (0,1) +(-0.05,-0.05) rectangle +(0.05,0.05)
    node[above right, font=\small, black] {$(0,1)$};
  \filldraw[fill=black!80, draw=black]
    (1,0) +(-0.05,-0.05) rectangle +(0.05,0.05)
    node[below right, font=\small, black] {$(1,0)$};
  \draw[-{Stealth}, blue!65, thick, line width=0.8pt]
    (0,0) -- +(-0.20, 0) node[above, font=\tiny, blue!65] {$\phi=\pi$};
  \draw[-{Stealth}, blue!65, thick, line width=0.8pt]
    (1,1) -- +(-0.20, 0) node[above, font=\tiny, blue!65] {$\phi=\pi$};
  \draw[-{Stealth}, blue!65, thick, line width=0.8pt]
    (0,1) -- +(0.20, 0) node[above, font=\tiny, blue!65] {$\phi=0$};
  \draw[-{Stealth}, blue!65, thick, line width=0.8pt]
    (1,0) -- +(0.20, 0) node[below, font=\tiny, blue!65] {$\phi=0$};
  \draw[rounded corners=2pt, gray!50]
    (-0.34, -0.32) rectangle (1.34, -0.15);
  \node[font=\footnotesize, anchor=west] at (-0.30, -0.235)
    {$\circ$ class $-1$ \quad $\blacksquare$ class $+1$ \quad
     {\color{blue!65}$\rightarrow$} fiber angle $\phi(x)$
     \quad \textit{gray}: class $-1$ region};
\end{tikzpicture}
\caption{XOR with $G=O(2)$. The decision boundary is $|x^1-x^2|=1/\sqrt{2}$.
The fiber angle $\phi$ equals $\pi$ at negative points and $0$ at positive points.
The section $s(x) = r(\cos\phi(x),\sin\phi(x))^\top$ has constant norm and never
vanishes; the sign change of $\sigma = s^1 = r\cos\phi$ is a fiber rotation.}
\label{fig:xor}
\end{figure}
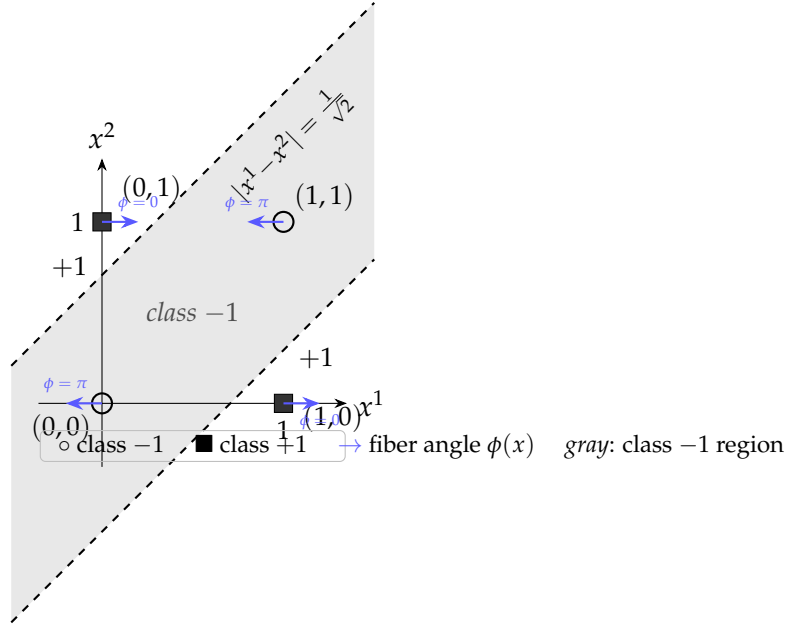

\subsection{The $O(2)$ connection}

We take $G = O(2)$, $E = \R^2\times\R^2$, $\Sig = \theta\cdot J$ with
$J = \bigl(\begin{smallmatrix}0&{-1}\\1&0\end{smallmatrix}\bigr)$.
Writing $s(x) = r\bigl(\cos\phi(x),\sin\phi(x)\bigr)^\top$, horizontality gives
$\theta = -d\phi$. With $\varphi = (1,0)$ and the ansatz
\begin{equation}\label{eq:phi-xor}
  \phi(x^1,x^2) = \pi\bigl(1-(x^1-x^2)^2\bigr),
\end{equation}
all four XOR points are correctly classified (verified by direct substitution). The
connection is $\Sig = 2\pi(x^1-x^2)(dx^1-dx^2)\cdot J$ with curvature $F = 0$,
since $(x^1-x^2)(dx^1-dx^2) = d[\tfrac{1}{2}(x^1-x^2)^2]$ is exact. The decision
boundary is $|x^1-x^2| = 1/\sqrt{2}$ (Figure~\ref{fig:xor}).

\begin{remark}[Local validity of the ansatz]
The angle function \eqref{eq:phi-xor} is unbounded: $\phi \to -\infty$ as
$|x^1-x^2| \to \infty$, so $\cos\phi$ oscillates indefinitely and the global
classifier produces infinitely many parallel sign-change stripes. This is a local
construction, valid near the four data points. To obtain a well-behaved global
classifier with exactly two regions, restrict to the strip
$|x^1-x^2| \leq 1/\sqrt{2} + \varepsilon$ for small $\varepsilon > 0$, or
replace the polynomial ansatz with a bounded angle function (such as the
$\tanh$-based construction of \S\ref{sec:flower}) that saturates to $0$ and $\pi$
far from the data.
\end{remark}

% ================================================================
\section{The Flower Problem with $G = O(2)$}\label{sec:flower}
% ================================================================

\subsection{Setup}

We consider the planar flower classification problem, a standard benchmark in
neural network pedagogy. This section studies \textbf{Problem B} (continuous boundary
representation): the decision boundary $\Gam$ is given analytically, and we
construct a connection-section pair that represents it exactly. This is distinct
from Problem A (finite labeled data), where only point evaluations are prescribed
and the boundary is determined by the interpolating solution. The data space is
$M = \R^2$, equipped with polar coordinates $(r, \alpha)$. The classification
problem has tuple
\[
  \bigl(M = \R^2,\; g = g_{\mathrm{Eucl}},\; E = M\times\R^2,\;
        G = O(2),\; \varphi = (1,0),\; \mathcal{D}\bigr).
\]
The decision boundary is a closed curve
with $k$-fold rotational symmetry:
\begin{equation}\label{eq:flower-boundary}
  \Gam:\quad r \;=\; r_0\bigl(1 + \varepsilon\cos(k\alpha)\bigr),
\end{equation}
where $r_0 > 0$ is the mean radius, $\varepsilon \in (0,1)$ controls the petal
amplitude, and $k$ is the number of petals. Points inside $\Gam$ belong to class
$+1$; points outside to class $-1$. We work with $k=4$, $r_0=1$, $\varepsilon=0.4$
throughout.

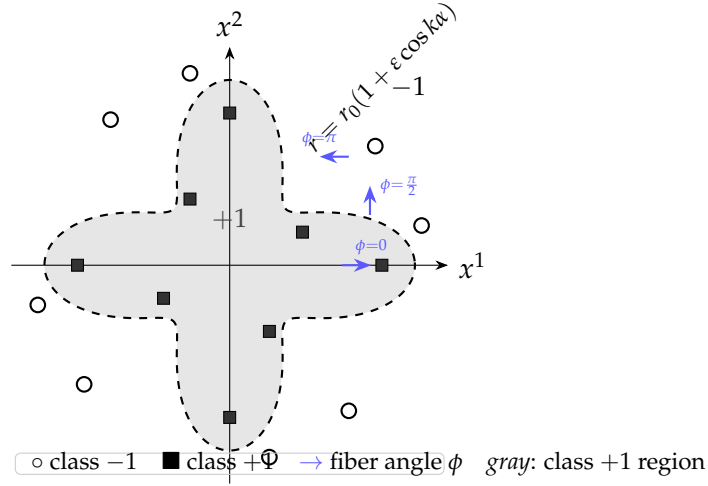
\begin{figure}[ht]
\centering
\begin{tikzpicture}[scale=1.75]
  % ---- Flower interior (class +1, light gray) ----
  \fill[gray!20]
    plot[domain=0:360, samples=300, smooth cycle, variable=\a]
    ({(1 + 0.4*cos(4*\a))*cos(\a)}, {(1 + 0.4*cos(4*\a))*sin(\a)});

  % ---- Flower boundary ----
  \draw[thick, black, dashed]
    plot[domain=0:360, samples=300, smooth cycle, variable=\a]
    ({(1 + 0.4*cos(4*\a))*cos(\a)}, {(1 + 0.4*cos(4*\a))*sin(\a)});

  % ---- Axes ----
  \draw[-{Stealth}] (-1.65,0) -- (1.65,0) node[right] {$x^1$};
  \draw[-{Stealth}] (0,-1.65) -- (0,1.65) node[above] {$x^2$};

  % ---- Class +1 points (filled squares, inside petals) ----
  % Near petal tips: alpha = 0, 90, 180, 270
  \foreach \px/\py in {1.15/0.0, 0.0/1.15, -1.15/0.0, 0.0/-1.15,
                        0.55/0.25, -0.3/0.5, -0.5/-0.25, 0.3/-0.5}
    \filldraw[fill=black!80, draw=black]
      (\px,\py) +(-0.045,-0.045) rectangle +(0.045,0.045);

  % ---- Class -1 points (open circles, outside / inter-petal) ----
  % Between petals (alpha = 45,135,225,315) at larger r
  \foreach \px/\py in {1.1/0.9, -0.9/1.1, -1.1/-0.9, 0.9/-1.1,
                        1.45/0.3, -0.3/1.45, -1.45/-0.3, 0.3/-1.45}
    \draw[thick] (\px,\py) circle (0.055);

  % ---- Fiber angle arrows ----
  % Inside petal (alpha=0, r=1.1): phi=0, arrow right
  \draw[-{Stealth}, blue!65, thick, line width=0.8pt]
    (0.85, 0) -- +(0.22, 0)
    node[above, font=\tiny, blue!65] {$\phi{=}0$};

  % Outside (alpha=45 deg, r=1.2): phi=pi, arrow left
  \draw[-{Stealth}, blue!65, thick, line width=0.8pt]
    (0.9, 0.82) -- +(-0.22, 0)
    node[above, font=\tiny, blue!65] {$\phi{=}\pi$};

  % On boundary near alpha=20 deg: phi=pi/2, arrow up
  \draw[-{Stealth}, blue!65, thick, line width=0.8pt]
    (1.06, 0.38) -- +(0, 0.22)
    node[right, font=\tiny, blue!65] {$\phi{=}\tfrac{\pi}{2}$};

  % ---- Region labels ----
  \node[font=\small\itshape, gray!55!black] at (0, 0.35)     {$+1$};
  \node[font=\small\itshape]               at (1.35, 1.35)   {$-1$};

  % ---- Boundary label ----
  \node[font=\footnotesize, rotate=45] at (1.15, 1.42)
    {$r = r_0(1+\varepsilon\cos k\alpha)$};

  % ---- Legend ----
  \draw[rounded corners=2pt, gray!40]
    (-1.62, -1.58) rectangle (1.62, -1.42);
  \node[font=\footnotesize, anchor=west] at (-1.58, -1.50)
    {$\circ$ class $-1$\quad $\blacksquare$ class $+1$\quad
     {\color{blue!65}$\rightarrow$} fiber angle $\phi$\quad
     \textit{gray}: class $+1$ region};
\end{tikzpicture}
\caption{Flower classification with $k=4$ petals, $r_0=1$, $\varepsilon=0.4$,
and $G=O(2)$. The shaded interior of the closed curve $\Gam$ is class $+1$;
the exterior is class $-1$. The fiber angle $\phi \approx 0$ inside petals,
$\phi \approx \pi$ outside, and $\phi = \pi/2$ on $\Gam$. The connection
1-form $\theta = -d\phi$ oscillates $k$ times around $\Gam$, with the
oscillation frequency encoding the number of petals.}
\label{fig:flower}
\end{figure}

\subsection{The angle function and the connection}

We construct the solution following the pattern of Theorem~\ref{thm:flat}. Define the signed
distance from the boundary:
\[
  P(r,\alpha) \;=\; r - r_0\bigl(1 + \varepsilon\cos(k\alpha)\bigr),
\]
so $P < 0$ inside $\Gam$ and $P > 0$ outside. The angle function is:
\begin{equation}\label{eq:phi-flower}
  \phi(r,\alpha) \;=\; \frac{\pi}{2}\bigl(1 + \tanh(\lambda P(r,\alpha))\bigr),
  \qquad \lambda > 0,
\end{equation}
taking values in $(0,\pi)$ with $\phi \approx 0$ deep inside and $\phi \approx \pi$
deep outside. The section is $s = r(\cos\phi, \sin\phi)^\top$ with constant norm
$r > 0$, and the classifier is $\sigma = r\cos\phi$.

Computing $dP$ in polar coordinates:
\[
  dP \;=\; dr \;-\; r_0\varepsilon k\sin(k\alpha)\,d\alpha.
\]
Setting $\beta(\rho) = \frac{\pi\lambda}{2}\,\mathrm{sech}^2(\lambda\rho)$,
the connection 1-form is:
\begin{equation}\label{eq:conn-flower}
  \theta \;=\; -d\phi \;=\; -\beta(P)\,dP
  \;=\; -\beta(P)\,dr \;+\; r_0\varepsilon k\,\beta(P)\sin(k\alpha)\,d\alpha.
\end{equation}

\subsection{Where $k$ is encoded}

The two components of $\theta$ carry different information:

\textbf{Radial component} $\theta_r = -\beta(P)$: controls how rapidly the
classifier transitions from $+1$ to $-1$ as we cross $\Gam$ radially. It is
concentrated near $\Gam$ (where $P \approx 0$) and decays exponentially away
from it, with width $1/\lambda$.

\textbf{Angular component} $\theta_\alpha = r_0\varepsilon k\,\beta(P)\sin(k\alpha)$:
encodes the shape of the boundary. Along the boundary $P = 0$, this simplifies to:
\[
  \theta_\alpha\big|_{\Gam} \;=\; \frac{\pi\lambda r_0\varepsilon k}{2}\sin(k\alpha).
\]
This oscillates exactly $k$ times as $\alpha$ traverses $[0,2\pi]$. The number
of petals appears as the \textbf{oscillation frequency of the angular component
of the connection 1-form along $\Gam$} --- but only for this particular choice
of angle function $\phi$.

\begin{remark}[$k$ as a property of the solution, not the problem]
For Problem B (continuous $\Gam$ given), the integer $k$ is determined by the
boundary curve and is a genuine geometric feature of $\Gam$. For Problem A
(finite labeled data), $k$ is \emph{not} an invariant of the classification
problem: a different angle function $\phi'$ satisfying the same Dirichlet
conditions $\phi'(x_i) \in \{0,\pi\}$ at the same data points may have a
different oscillation count. The flower example works with Problem B, where
$\Gam$ is given analytically and $k$ is prescribed by the boundary's symmetry.
\end{remark}

\subsection{Flatness}

Since $\phi$ is a single-valued smooth function on $\R^2$:
\[
  F \;=\; d\Sig \;=\; -d^2\phi\cdot J \;=\; 0.
\]
The connection is flat, as guaranteed by Theorem~\ref{thm:flat}.

Integrating $\theta$ around any loop $\gamma$ in $\R^2$:
\[
  \oint_\gamma \theta \;=\; -\oint_\gamma d\phi \;=\; 0,
\]
since $\phi$ is single-valued. The holonomy is trivial for every loop. In
particular, the holonomy around $\Gam$ itself is zero --- the section returns
to its starting angle after going around the boundary once.

\subsection{Comparison with XOR and the role of topology}

\begin{center}
\renewcommand{\arraystretch}{1.25}
\begin{tabular}{@{}llll@{}}
\toprule
& \textbf{XOR} & \textbf{Flower} \\
\midrule
$\Gam$ topology    & Two lines (non-compact)   & Single closed curve $\cong S^1$ \\
Symmetry           & $\mathbb{Z}_2\times\mathbb{Z}_2$ & $\mathbb{Z}_k$ \\
Angle function     & Polynomial in $x^1-x^2$   & $\tanh$ of signed distance \\
Encoding of $k$    & Not present               & Frequency of $\theta_\alpha$ along $\Gam$ \\
Holonomy           & Trivial                   & Trivial \\
Curvature          & $F=0$                     & $F=0$ \\
\bottomrule
\end{tabular}
\end{center}

\medskip
Both examples are flat, consistent with Theorem~\ref{thm:flat}. The difference is purely
in the local geometry of the connection: the XOR connection has no preferred
angular direction, while the flower connection oscillates with frequency $k$ along
$\Gam$. Over $\R^2$, the integer $k$ is a \emph{metric} invariant of the solution,
not a topological one.

\begin{remark}[Why $O(2)$ appears in textbook examples]\label{rem:textbook}
The appearance of $G = O(2)$ in both the XOR and flower examples is not a
coincidence of the framework --- it reflects the way such examples are
\emph{constructed}. Textbook classification benchmarks are typically generated
algorithmically by the authors: XOR is built from reflections across the diagonals
of the unit square; the flower is built by rotating a radial template $k$ times.
Both constructions use rotations and reflections, which are exactly the elements
of $O(2)$. The symmetry group of the data generation process is a subgroup of
$O(2)$, and so the structure group $O(2)$ fits the problem perfectly.

In real classification problems --- sentiment analysis, fraud detection, medical
diagnosis --- the data is not generated by rotations. There is no reason for the
relevant structure group to be $O(2)$ or any other rotation group. The symmetry
group of the data distribution may be discrete, non-abelian, high-dimensional, or
difficult to identify. Choosing the right $G$ for a real problem requires
understanding the symmetries of the data --- a point we return to in \S\ref{sec:symmetry}.
\end{remark}

\begin{remark}[Bridge to the sphere]\label{rem:bridge}
The transition from $\R^2$ to a non-contractible base space $M$ changes the
status of $k$ fundamentally. On $S^2$, oriented rank-2 real bundles with structure
group $SO(2)$ are classified by $H^2(S^2, \mathbb{Z}) = \mathbb{Z}$,
with the integer being the Euler class $e(E) = k$. Any section of the degree-$k$
bundle must satisfy $\int_{S^2} F = 2\pi k$, making $k$ a quantised topological
charge rather than a local geometric feature. The flower example over $\R^2$ can
therefore be seen as the flat, degenerate limit of a richer family of classifiers
over $S^2$, where the petal count becomes a topological invariant of the bundle.
\end{remark}

% ================================================================
\section{Curvature of the Induced Connection}\label{sec:curvature}
% ================================================================

We now investigate precisely when the connection induced by a section is flat.
This leads to an explicit curvature formula, a counterexample to the naive
flatness conjecture, and a positive theorem showing that every classification
problem admits a flat solution.

\subsection{The curvature formula}

Let $s: M \to \R^m$ be smooth with $s(x) \neq 0$ everywhere. Horizontality
$\nabla_\Sig s = ds + \Sig\, s = 0$ constrains $\Sig$ only through
$\Sig\, s = -ds$: it determines $\Sig$ on the line spanned by $s$ and leaves the
action on $s^\perp$ free. A canonical choice is the rank-one (projector)
connection
\[
  \Sig_\mu \;=\; -(\del_\mu s)\,s^\dagger
  \;=\; -\frac{(\del_\mu s)\,s^T}{\norm{s}^2},
\]
which we analyse in this section.

\begin{remark}[Structure group of the projector connection]\label{rem:glm}
The choice above is a $\GL(m)$-connection, not an $O(m)$-connection: even for
constant-norm $s = r(\cos\phi,\sin\phi)^\top$ one computes
$\Sig + \Sig^T \neq 0$, so $\Sig \notin \so(m)$. The $O(2)$-connection used in
\S\ref{sec:xor}--\S\ref{sec:flower}, namely $\Sig = -d\phi \cdot J$ with
$J$ the standard complex structure, is a \emph{different} horizontal choice for
the same section. Both are flat in the constant-norm planar case, so the worked
examples are unaffected; but the general curvature analysis of this section
concerns the projector connection and is $\mathfrak{gl}(m)$-valued. Where a
$G$-connection is required, $\Sig$ must be specified as $\mathfrak{g}$-valued
from the outset.
\end{remark}

Write $a_\mu = \del_\mu s$, $a_{\mu\nu} = \del_\mu\del_\nu s$, and
$\alpha_\mu = s^\dagger a_\mu = \frac{s^T a_\mu}{\norm{s}^2} \in \R$.

Computing $\del_\mu\Sig_\nu$:
\[
  \del_\mu\Sig_\nu
  = -\frac{a_{\mu\nu}s^T + a_\nu a_\mu^T}{\norm{s}^2}
    + \frac{2\alpha_\mu\,a_\nu s^T}{\norm{s}^2}.
\]
The antisymmetric part (with $a_{\mu\nu}$ terms cancelling by symmetry of mixed
partials):
\[
  \del_\mu\Sig_\nu - \del_\nu\Sig_\mu
  = \frac{1}{\norm{s}^2}
    \bigl(a_\mu a_\nu^T - a_\nu a_\mu^T
          + 2\alpha_\mu a_\nu s^T - 2\alpha_\nu a_\mu s^T\bigr).
\]
The commutator term:
\[
  [\Sig_\mu,\Sig_\nu]
  = \frac{1}{\norm{s}^2}
    \bigl(\alpha_\nu a_\mu s^T - \alpha_\mu a_\nu s^T\bigr).
\]
Adding these and defining the projection onto $s^\perp$:
\[
  a_\mu^\perp \;=\; a_\mu - \alpha_\mu s
  \;=\; \bigl(I - s\,s^\dagger\bigr)a_\mu,
\]
the curvature reduces to:

\begin{equation}\label{eq:curvature}
  \boxed{F_{\mu\nu}
  \;=\; \frac{1}{\norm{s}^2}
        \Bigl(a_\mu\,(a_\nu^\perp)^T - a_\nu\,(a_\mu^\perp)^T\Bigr).}
\end{equation}

\subsection{When is the curvature zero?}

\begin{proposition}\label{prop:flat-criterion}
Write $a_\mu = a_\mu^\perp + \alpha_\mu s$ with $a_\mu^\perp \in s^\perp$. The
induced connection \eqref{eq:curvature} is flat if and only if, at every
$x \in M$ and for all $\mu,\nu$, \emph{both}
\[
  \text{(a)}\quad a_\mu^\perp \wedge a_\nu^\perp = 0,
  \qquad\qquad
  \text{(b)}\quad \alpha_\mu\, a_\nu^\perp = \alpha_\nu\, a_\mu^\perp .
\]
No condition on the topology of $M$ is required.
\end{proposition}

\begin{proof}
Substituting $a_\mu = a_\mu^\perp + \alpha_\mu s$ into \eqref{eq:curvature},
\[
  \norm{s}^2 F_{\mu\nu}
  \;=\; \underbrace{\Bigl(a_\mu^\perp (a_\nu^\perp)^T - a_\nu^\perp (a_\mu^\perp)^T\Bigr)}_{\in\; s^\perp \otimes s^\perp}
  \;+\; \underbrace{s\,\bigl(\alpha_\mu a_\nu^\perp - \alpha_\nu a_\mu^\perp\bigr)^T}_{\in\; s \otimes s^\perp}.
\]
Since $\R^m = \R s \oplus s^\perp$, the two summands lie in independent subspaces
of $\R^m\otimes\R^m$ and must vanish separately. The first vanishes exactly when
the $a_\mu^\perp$ are pairwise parallel, which is (a); the second exactly when
(b) holds. The criterion is pointwise and involves only the first derivative of
$s$, so the topology of $M$ is irrelevant.
\end{proof}

\begin{corollary}\label{cor:flat-constant-norm}
If $\norm{s}$ is constant then $\alpha_\mu \equiv 0$, condition (b) is vacuous,
and the connection is flat if and only if the $a_\mu$ span at most a line at each
point. In particular, if $\norm{s}$ is constant and the image of $s$ lies in a
$2$-dimensional linear subspace of $\R^m$, the connection is flat.
\end{corollary}

\begin{proof}
Constant norm gives $s^T a_\mu = 0$, hence $\alpha_\mu = 0$ and
$a_\mu^\perp = a_\mu$; (b) holds trivially and (a) becomes the stated span
condition. If moreover the image lies in a $2$-plane $V$, then $s(x) \in V$ and
$a_\mu(x) \in V$, so $\{a_\mu(x)\} \subset V \cap s(x)^\perp$, of dimension at
most $1$.
\end{proof}

\begin{remark}[Both conditions are needed]
Condition (a) alone is not sufficient. For $s(x^1,x^2) = (x^1,x^2,0)$ --- image
in the plane $\{x^3=0\}$, non-constant norm --- one computes, with
$\rho^2 = (x^1)^2+(x^2)^2$ and $v = (x^2,-x^1,0)$,
\[
  a_1^\perp = \tfrac{x^2}{\rho^2}\,v, \qquad
  a_2^\perp = -\tfrac{x^1}{\rho^2}\,v, \qquad
  (\alpha_1,\alpha_2) = \tfrac{1}{\rho^2}(x^1,x^2).
\]
The $a_\mu^\perp$ are parallel, so (a) holds; but
$\alpha_1 a_2^\perp - \alpha_2 a_1^\perp = -\rho^{-2} v \neq 0$, so (b) fails,
and indeed
\[
  F_{12} \;=\; \frac{1}{\rho^4}
  \begin{pmatrix} -x^1x^2 & (x^1)^2 & 0 \\
                  -(x^2)^2 & x^1x^2 & 0 \\
                  0 & 0 & 0 \end{pmatrix} \;\neq\; 0 .
\]
Conversely, when $\dim M = 1$ there are no $2$-forms and every section is flat
regardless of its image. Constant norm is what makes (b) automatic; all sections
used in this paper ($s = r(\cos\phi,\sin\phi)^\top$) have constant norm.
\end{remark}

\subsection{Counterexample: a section mapping to $S^2$}

Take $M = \R^2$ and $s: M \to S^2 \subset \R^3$ defined by
\[
  s(x^1,x^2) = \bigl(\sin x^1\cos x^2,\;\sin x^1\sin x^2,\;\cos x^1\bigr).
\]
This has constant norm $\norm{s} = 1$. The derivatives $a_1 = \del_1 s$ and
$a_2 = \del_2 s$ both lie in $s^\perp$ and are linearly independent (they are the
coordinate tangent vectors on $S^2$). Therefore $\{a_\mu^\perp\}$ spans a
2-dimensional subspace of $s^\perp$, and $F_{12} \neq 0$.

Explicitly, $F_{12} = a_1 a_2^T - a_2 a_1^T$, which is proportional to the area
form of $S^2$. This is the curvature of the \emph{tautological bundle} over $S^2$
--- a classical non-trivial $O(3)$-bundle. The original flatness conjecture is
therefore \textbf{false} for $m \geq 3$.

\subsection{Every classification problem admits a flat solution}

The counterexample shows that not every section induces a flat connection. However,
for any classification problem there always \emph{exists} a section inducing a
flat connection.

\begin{theorem}[Existence of flat solutions]\label{thm:flat}
Let $M$ be a smooth manifold and $\mathcal{D} = \{(x_i, y_i)\}$ a finite binary
classification dataset. Then there exists a section $s: M \to \R^2
\hookrightarrow \R^m$ of constant norm and a flat $O(2)$-connection $\nabla$ on
$E = M \times \R^m$ such that $[\nabla, s]$ solves the classification problem
with readout $\varphi = (1,0,\ldots,0)$.
\end{theorem}

\begin{proof}
By a standard partition of unity argument, there exists a smooth function
$f: M \to [-1,1]$ with $f(x_i) = y_i$ for all $i$. Define
\[
  \theta(x) = \frac{\pi}{2}(1 - f(x)) \;\in\; [0,\pi],
\]
so that $f(x_i) = +1 \Rightarrow \theta(x_i) = 0$ and $f(x_i) = -1 \Rightarrow
\theta(x_i) = \pi$. Set
\[
  s(x) = r\bigl(\cos\theta(x),\;\sin\theta(x),\;0,\ldots,0\bigr)^T \in \R^m
\]
for any $r > 0$. The section takes values in a fixed 2-dimensional subspace, so
by Proposition~\ref{prop:flat-criterion} the induced connection is flat. The
horizontality condition gives $\Sig = -d\theta \cdot J$ where $J$ acts on the
first two components. The curvature is
\[
  F = d\Sig + \Sig\wedge\Sig = -d^2\theta\cdot J + (d\theta\wedge d\theta)\cdot J^2 = 0,
\]
since $d^2 = 0$ and $d\theta\wedge d\theta = 0$ for any 1-form. The classifier is
$\varphi(s(x)) = r\cos\theta(x)$, which has the correct sign at every data point
since $\cos 0 = +1$ and $\cos\pi = -1$.
\end{proof}

\begin{corollary}\label{cor:O2-sufficient}
Every finite binary classification problem on any smooth manifold $M$ can be
solved with a flat $O(2)$-connection. In particular, for finite data:
\begin{enumerate}[label=(\roman*), itemsep=2pt]
  \item Curvature is never necessary;
  \item The problem always reduces to the $O(2)$ case, regardless of the ambient
        fiber dimension $m$;
  \item The non-trivial geometry (when $m \geq 3$) encodes topological information
        about the solution as a map $s: M \to S^{m-1}$, not additional expressive
        power over $O(2)$.
\end{enumerate}
\end{corollary}

\begin{remark}
No condition on the topology of $M$ is required. The partition of unity
argument produces a smooth interpolating function on any smooth manifold;
$d^2\theta = 0$ and $d\theta\wedge d\theta = 0$ hold by the elementary
properties of the exterior derivative, independent of the topology of $M$;
and the bundle $E = M\times\R^m$ is trivial by construction. The topological
obstruction described in \S\ref{sec:sphere} --- where flat solutions can fail to
exist --- arises only when one fixes a \emph{non-trivial} bundle over a
non-contractible base, a situation outside the scope of this theorem.
\end{remark}

\begin{remark}[What the connection contributes, and where]
\label{rem:connection-derived}
It is worth stating plainly what Theorem~\ref{thm:flat} and
Corollary~\ref{cor:O2-sufficient} do \emph{not} say. Throughout
\S\ref{sec:xor}--\S\ref{sec:curvature} the connection is not an independent
variable: horizontality $\Sig\,s = -ds$ determines $\Sig$ from $s$ (up to the
freedom on $s^\perp$ noted in Remark~\ref{rem:glm}), so a solution is specified
by the section alone, equivalently by the angle function $\phi$. Any nowhere-zero
$s = r(\cos\phi,\sin\phi)^\top$ with any smooth $\phi$ is a flat $O(2)$ solution,
and the classifier is $r\cos\phi$. Theorem~\ref{thm:flat} therefore carries no
bundle-theoretic content beyond the partition-of-unity fact that a smooth
function with prescribed signs at finitely many points exists, and
Corollary~\ref{cor:O2-sufficient} is automatic for the same reason. The
nonlinearity that classical machine learning attributes to the activation
function has been \emph{relocated} to $\phi$, not derived from the geometry.

The framework acquires genuine content in three places, none of them here.
\begin{enumerate}[label=(\roman*), itemsep=2pt]
  \item \textbf{Topology of the base.} When $M$ is non-contractible, the bundle
        constrains which $\phi$ are admissible: the parity invariant on $S^1$
        and the Euler class on $S^2$ (\S\ref{sec:sphere}) are statements no
        choice of $\phi$ can evade.
  \item \textbf{Energy selection.} Imposing $\mathcal{E}_L$ replaces the
        arbitrary $\phi$ by a distinguished one (\S\ref{sec:yang-mills}); this
        is where the Green's function, the capacitance system, and the
        connection to RKHS interpolation live.
  \item \textbf{The connection as an independent field.} Only when $A$ carries
        its own energy $\int\|F_A\|^2$ and its own topological constraint,
        independently of the section, does the gauge-theoretic structure do
        work. That is the Yang--Mills--Higgs setting of the companion paper
        \cite{Vasii2026Paper2}, not of this one.
\end{enumerate}
In this paper the connection is largely a bookkeeping device: it makes (i) and
(ii) expressible in a uniform language and supplies the dictionary of
\S\ref{sec:dictionary}. A reader should not expect gauge-theoretic content
where the connection is merely derived from the section.
\end{remark}

\begin{remark}[Role of curvature]
If curvature is never necessary for finite data, what does it encode?
The curvature of the induced connection is the pullback of the canonical
curvature of the tautological bundle over $S^{m-1}$: it measures how the
section $s\colon M \to S^{m-1}$ wraps around the sphere. Among all solutions
to a given classification problem, the flat $O(2)$ solution minimises the
curvature functional $\int_M\norm{F}^2\,\mathrm{dvol}$. Solutions with non-zero
curvature are geometrically richer and may carry better regularity or
generalisation properties --- a question connected to the Yang-Mills variational
problem \cite{atiyah1983} (Q3 in \S\ref{sec:open}).
\end{remark}

% ================================================================
\section{Text Classification on Spheres}\label{sec:sphere}
% ================================================================

\subsection{Embeddings and cosine similarity}

In modern natural language processing, documents are represented as vectors in
$\R^n$ called \emph{embeddings}. Comparing documents by \emph{cosine similarity}
--- the inner product of their normalised versions --- is equivalent to comparing
their projections onto the unit sphere $S^{n-1} \subset \R^n$. Binary
classification (positive/negative sentiment, topic A vs topic B) is then a
problem over $M = S^{n-1}$, not over $\R^n$.

This changes the geometry fundamentally. Unlike $\R^n$, which is contractible,
$S^{n-1}$ is non-contractible and supports non-trivial vector bundles.
\S\ref{sec:curvature} established that over any smooth manifold every finite
classification problem admits a flat $O(2)$ solution (Theorem~\ref{thm:flat}).
Over spheres this can fail: certain bundles have non-zero characteristic classes
that force non-zero curvature and impose topological constraints on the decision
boundary independent of any training procedure.

The dimension $n$ determines which obstructions arise. We develop three cases.

\subsection{$M = S^1$: circular semantics and the Möbius classifier}

\paragraph{Setting.}
$S^1$ arises for 2-dimensional normalised embeddings, and more generally models
circular semantic spaces: a political spectrum that wraps around, a periodic
temporal structure, or any domain where the two extremes are identified. With
coordinate $\alpha \in [0, 2\pi)$, the data space is the circle. The
classification problem has tuple
\[
  \bigl(M = S^1,\; g = g_{\mathrm{round}},\; E \to S^1,\;
        G = O(2),\; \varphi = (1,0),\; \mathcal{D}\bigr),
\]
where $g_{\mathrm{round}} = d\alpha^2$ is the round metric on $S^1$ and $E$ is
the (necessarily trivial) rank-2 bundle; the topological content lies in the
line subbundle $L_s \subset E$, as explained below.

$S^1$ is the only sphere that is not simply connected: $\pi_1(S^1) = \mathbb{Z}$.
Flat connections are classified by holonomy representations $\rho: \mathbb{Z} \to G$,
determined by the single element $\rho(1) \in G$ --- the fiber rotation after one
full traversal of the circle. For $G = O(2)$, this is $R(\Delta\phi) \in O(2)$
where $\Delta\phi = \phi(2\pi) - \phi(0)$ is the total rotation of the fiber angle.

\paragraph{The rank-2 bundle is always trivial; the invariant lives elsewhere.}
Rank-2 real bundles over $S^1$ are classified by $\pi_0(O(2)) = \mathbb{Z}_2$,
detected by whether the holonomy lies in $SO(2)$ or in the reflection component.
Since $R(\Delta\phi) \in SO(2)$ for every $\Delta\phi$ --- in particular
$R(\pi) = -I \in SO(2)$ --- \emph{every} bundle arising from a continuous fiber
angle $\phi$ is trivial as a rank-2 bundle. A genuinely non-orientable rank-2
bundle over $S^1$ would require a reflection as holonomy, which no such $\phi$
produces.

The topological content is instead carried by the \textbf{real line subbundle}
$L_s = \mathrm{span}\,\{s(x)\} \subset E$, equivalently the line field
$x \mapsto [s(x)] \in \mathbb{RP}^1$. Its first Stiefel--Whitney class
$w_1(L_s) \in H^1(S^1;\mathbb{Z}_2) \cong \mathbb{Z}_2$ is the degree mod $2$
of that line field, and it is a genuine invariant:
$w_1(L_s) = 0$ iff $\Delta\phi \in 2\pi\mathbb{Z}$, and $w_1(L_s) \neq 0$ iff
$\Delta\phi \in \pi + 2\pi\mathbb{Z}$.

\textbf{Orientable line field} ($\Delta\phi \in 2\pi\mathbb{Z}$, $w_1(L_s) = 0$):
the section $s(\alpha) = r(\cos\phi(\alpha), \sin\phi(\alpha))^\top$ is
single-valued and $\sigma = r\cos\phi$ has an \emph{even} number of sign changes
(including zero).

\textbf{Möbius line field} ($\Delta\phi = \pi$, $w_1(L_s) \neq 0$): the simplest
non-trivial case. The flat connection and explicit section are:
\[
  \theta = -\tfrac{1}{2}\,d\alpha, \qquad
  \phi(\alpha) = \tfrac{\alpha}{2} + \phi_0, \qquad
  \sigma(\alpha) = r\cos\!\bigl(\tfrac{\alpha}{2} + \phi_0\bigr).
\]
With $\phi_0 = 0$: $\sigma$ has exactly \emph{one} zero on $[0,2\pi)$, at
$\alpha = \pi$. Curvature: $F = d\theta \cdot J = 0$ (flat). Here $s$ is
double-valued as a map to $\R^2$ but the line $[s]$ is single-valued: the line
field closes up after one traversal having reversed orientation, which is
precisely $w_1(L_s) \neq 0$.

\begin{remark}[Why the parity is not gauge-trivial]
The connection $\theta = -\tfrac12 d\alpha$ is gauge-equivalent to $0$ on the
universal cover via $R(\alpha/2)$, under which $s$ becomes constant and $\sigma$
acquires no zeros. But that gauge transformation is not single-valued on $S^1$:
$R(\alpha/2)$ returns to $-I$, not $I$, after $\alpha \mapsto \alpha + 2\pi$.
The parity of the zeros of $\sigma$ is therefore not an invariant of the rank-2
bundle $E$ (which is trivial) but of the pair $(L_s, \varphi)$: the line field
together with the fixed readout direction.
\end{remark}

\begin{figure}[ht]
\centering
\begin{tikzpicture}[scale=1.6]
  \def\R{1.2}
  % Positive arc: alpha in [0, pi] (upper semicircle), sigma = cos(alpha/2) > 0
  \draw[line width=5pt, gray!22, line cap=butt]
    (0:\R) arc[start angle=0, end angle=180, radius=\R];
  % Negative arc: alpha in [pi, 2pi] (lower semicircle), sigma < 0
  \draw[line width=5pt, gray!58, line cap=butt]
    (180:\R) arc[start angle=180, end angle=360, radius=\R];
  % Circle outline
  \draw[thick] (0,0) circle (\R);
  % Decision boundary point at alpha=pi (left)
  \filldraw[black] (180:\R) circle (0.072)
    node[left=3pt, font=\small] {$\sigma=0$};
  % Fiber angle arrows
  \draw[-{Stealth}, blue!65, thick, line width=0.85pt]
    (0:\R)   -- ++(0.28, 0)    node[right,       font=\scriptsize, blue!65] {$\phi=0$};
  \draw[-{Stealth}, blue!65, thick, line width=0.85pt]
    (90:\R)  -- ++(0.20, 0.20) node[above right, font=\scriptsize, blue!65] {$\phi=\tfrac{\pi}{4}$};
  \draw[-{Stealth}, blue!65, thick, line width=0.85pt]
    (180:\R) -- ++(0, 0.28)    node[above,       font=\scriptsize, blue!65] {$\phi=\tfrac{\pi}{2}$};
  \draw[-{Stealth}, blue!65, thick, line width=0.85pt]
    (270:\R) -- ++(-0.20, 0.20)node[below left,  font=\scriptsize, blue!65] {$\phi=\tfrac{3\pi}{4}$};
  % Alpha labels
  \node[font=\footnotesize] at (1.62,  0)   {$\alpha=0$};
  \node[font=\footnotesize] at (0,  1.55)   {$\alpha=\tfrac{\pi}{2}$};
  \node[font=\footnotesize] at (0, -1.55)   {$\alpha=\tfrac{3\pi}{2}$};
  % Class labels
  \node[font=\small\itshape] at (0,  0.58)  {$\sigma > 0$};
  \node[font=\small\itshape] at (0, -0.58)  {$\sigma < 0$};
  % Annotation
  \node[font=\footnotesize] at (0, -2.1)
    {M\"{o}bius line field: $\Delta\phi = \pi$, $w_1(L_s) \neq 0$};
\end{tikzpicture}
\caption{Möbius line field on $S^1$ with $\phi(\alpha) = \alpha/2$. The fiber
angle increases by $\pi$ around the circle, so the line $[s]$ closes up with
reversed orientation: $w_1(L_s) \neq 0$. The
classifier $\sigma = r\cos(\alpha/2)$ has exactly one sign change at $\alpha =
\pi$ (marked dot). Light arc: $\sigma > 0$; dark arc: $\sigma < 0$. The
connection $\theta = -\frac{1}{2}d\alpha$ is flat, and the ambient rank-2
bundle is trivial.}
\label{fig:s1}
\end{figure}

\begin{proposition}\label{prop:s1-parity}
On $S^1$ with structure group $O(2)$ and fixed readout $\varphi$, the parity of
the number of zeros of $\sigma = \varphi \circ s$ equals
$w_1(L_s) \in H^1(S^1;\mathbb{Z}_2)$, the first Stiefel--Whitney class of the
real line subbundle $L_s = \mathrm{span}\{s\}$:
\begin{itemize}[itemsep=2pt]
  \item Even number of zeros (including zero) $\;\Leftrightarrow\;$
        $w_1(L_s) = 0$, orientable line field.
  \item Odd number of zeros $\;\Leftrightarrow\;$
        $w_1(L_s) \neq 0$, Möbius line field.
\end{itemize}
Both cases admit flat connections, and in both the ambient rank-2 bundle $E$ is
trivial. The Möbius line field achieves the minimum: one zero of $\sigma$, one
decision boundary point.
\end{proposition}

\begin{remark}[Functor language in the companion paper]
The Möbius line field is exactly the non-trivial functor
$C: \Pi_1(S^1) \to B(\mathbb{Z}_2)$ of the companion paper \cite{Vasii2026Paper2}:
the fundamental groupoid of $S^1$ maps to the classifying space of $\mathbb{Z}_2$,
and the monodromy class $[C] \in H^1(S^1, \mathbb{Z}_2) \cong \mathbb{Z}_2$
coincides with $w_1(L_s)$ above. The line bundle $L_s$ is Paper~2's $L_\xi$.
\end{remark}

\subsection{$M = S^2$: the Dirac monopole and forced curvature}

\paragraph{Setting and structure group.}
$S^2$ arises for 3-dimensional normalised embeddings. A real binary classifier
$f: S^2 \to \R$ lives in a real line bundle with $G = \R^*$; real line bundles
over $S^2$ are classified by $w_1 \in H^1(S^2;\mathbb{Z}/2) = 0$ — they are all
trivial and carry no topological obstruction.

The topologically interesting case uses the rank-2 real bundle setup of the rest
of the paper ($G = O(2)$, fiber $\R^2$, readout $\varphi = (1,0)$). Restricting
to orientation-preserving gauge transformations gives $G = SO(2)$
acting on $\R^2$. Such oriented rank-2 real bundles over $S^2$ are classified by
the \textbf{Euler class} $e(E) \in H^2(S^2,\mathbb{Z}) = \mathbb{Z}$. The
classification problem has tuple
\[
  \bigl(M = S^2,\; g = g_{\mathrm{round}},\; E_k \to S^2,\;
        G = SO(2),\; \varphi = (1,0),\; \mathcal{D}\bigr),
\]
where $E_k$ is the oriented rank-2 real bundle with Euler class $k$.

\paragraph{Curvature is topologically forced.}
For $e = k \neq 0$, any connection $\nabla$ on $E_k$ satisfies:
\[
  \int_{S^2} F_\nabla \;=\; 2\pi k \;\neq\; 0.
\]
If $F = 0$ everywhere, the integral would vanish --- contradicting $k \neq 0$.
Therefore \textbf{no flat connection exists on $E_k$ for $k \neq 0$}.
This is the first case in the framework where Theorem~\ref{thm:flat} does not
apply: the flat $O(2)$ reduction is topologically obstructed.

\paragraph{The Dirac monopole connection.}
The canonical connection on $E_k$ is the \emph{Dirac magnetic monopole}. In
spherical coordinates $(\theta, \psi)$ ($\theta \in [0,\pi]$ polar,
$\psi \in [0,2\pi)$ azimuthal):
\begin{align*}
  A^N &= \frac{k}{2}(1-\cos\theta)\,d\psi \quad
  \text{(northern chart, regular away from south pole)}, \\
  A^S &= -\frac{k}{2}(1+\cos\theta)\,d\psi \quad
  \text{(southern chart, regular away from north pole)}.
\end{align*}
On the equatorial overlap: $A^N - A^S = k\,d\psi$, an $SO(2)$ gauge
transformation with transition function $R(k\psi) \in SO(2)$ (rotation by $k\psi$).
The curvature is:
\[
  F = dA^N = \frac{k}{2}\sin\theta\,d\theta\wedge d\psi
    = \frac{k}{2}\,\omega_{S^2},
\]
proportional to the area form of $S^2$. Integrating confirms:
$\int_{S^2} F = \frac{k}{2}\cdot 4\pi = 2\pi k$.

\paragraph{Forced vanishing of the section and the decision boundary.}
By the Poincaré--Hopf theorem, for a section $s: S^2 \to E_k$ with isolated
zeros,
\[
  \sum_{x_0 \,:\, s(x_0)=0} \mathrm{ind}_{x_0}(s) \;=\; e(E_k) \;=\; k ,
\]
the sum of local indices. For $k \neq 0$ the zero set is therefore non-empty,
though its cardinality can be as small as one (a single zero of index $k$).
Here ``zero of $s$'' means $s(x_0) = \mathbf{0} \in \R^2$ --- the section
\emph{vanishes completely}, not merely $s^1(x_0) = 0$. This is real
codimension 2: isolated points, not curves.

The decision boundary $\Gam = \{x : s^1(x) = 0\} = \{x : \varphi(s(x)) = 0\}$
is still a curve (codimension 1), as required for binary classification.
The relationship is: at a zero $x_0$ of $s$ (where $s(x_0) = \mathbf{0}$),
necessarily $s^1(x_0) = 0$, so $x_0 \in \Gam$. The forced zeros of $s$ are
therefore forced \textbf{points on the decision boundary curve}: the topology
forces the boundary to pass through the zero set, but the boundary itself
remains a curve.

\begin{remark}[The hairy ball theorem as a learnability obstruction]
The tangent bundle $TS^2$ has Euler class $\chi(S^2) = 2$: every smooth vector
field on $S^2$ must vanish somewhere, and the zeros carry \emph{total index} $2$
(the hairy ball theorem). The count is by index, not by cardinality: a single
zero of index $2$ is possible, as for the stereographic pullback of a constant
field. In our framework: if the bundle for a classification problem on $S^2$ has
$e(E_k) = 2$, every section $s$ must vanish with total index $2$, so the decision
boundary must pass through the zero set regardless of training. These are
topological constraints on the classifier's geometry, not limitations of the
optimiser.
\end{remark}

\paragraph{Bridge from the flower.}
In \S6, the petal count $k$ of the flower classifier over $\R^2$ was a local
metric invariant of the connection 1-form --- the oscillation frequency of
$\theta_\alpha$ along $\Gam$. Over $S^2$, the same integer $k = e(E_k) \in \mathbb{Z}$
is a global topological invariant of the bundle, quantised and independent of
any local choice. The promotion from metric to topological is the key effect of
passing from a contractible to a non-contractible base space.

\subsection{Implications for learnability}

The sphere analysis extends the framework in three directions.

\textbf{Curvature can be necessary.} Over $\R^n$ and over $S^1$,
every finite classification problem has a flat solution. Over $S^2$, for Problem A
(finite data), Theorem~\ref{thm:flat} still applies: the trivial bundle always
admits a flat $O(2)$ solution. The topological obstruction bites in two situations:
(i) Problem B, where a continuous boundary $\Gam$ with non-zero winding requires a
non-trivial bundle; or (ii) when a specific non-trivial bundle $E_k$ is
\emph{chosen} as the geometric model — then $e \neq 0$ forces curvature and the
flat reduction is obstructed.

\textbf{Forced points on the decision boundary.} For a chosen degree-$k$ bundle
over $S^2$ with $e = k \neq 0$, any section must vanish completely at $|k|$
points, which forces the decision boundary curve to pass through those points.
This is a topological constraint on the geometry of the classifier given the
choice of bundle, not a prohibition on binary classification itself.

\textbf{Parity constraints on $S^1$.} The number of decision boundary points on
a circular semantic space must have a definite parity determined by the bundle.
Attempting to train a classifier with the wrong parity will always fail.

\medskip
\noindent\textit{Arc 1 ends here.} Sections~\ref{sec:hierarchy}--\ref{sec:tda}
(Arc~2) address the question left open by the examples: which structure group
should be used, and how should it be selected from the data?

% ================================================================
\section{The Structure Group Hierarchy}\label{sec:hierarchy}
% ================================================================

The choice of $G$ is not arbitrary. Topological obstructions (characteristic
classes of $E$) may prevent certain reductions. Beyond obstructions, the choice
has direct geometric meaning:

\begin{center}
\renewcommand{\arraystretch}{1.3}
\begin{tabular}{@{}lll@{}}
\toprule
\textbf{Structure group $G$} & \textbf{Fiber geometry} & \textbf{Constraint on $s$} \\
\midrule
$\R^*$ (abelian)  & Scaling            & $s$ must vanish on $\Gam$ \\
$O(2)$            & Rotation in plane  & $\norm{s}$ constant, values in $\R^2$ \\
$O(m)$            & Riemannian metric  & $\norm{s}$ constant, $s\neq 0$ \\
$U(m)$            & Hermitian metric   & $\norm{s}$ constant, complex structure \\
$\mathrm{SL}(m)$  & Volume form        & $\det$ of fiber frame fixed \\
$\mathrm{Sp}(2m)$ & Symplectic form    & Symplectic structure on fiber \\
$\GL(m)$          & None               & No constraint \\
\bottomrule
\end{tabular}
\end{center}

\medskip
By Theorem~\ref{thm:flat} and Corollary~\ref{cor:O2-sufficient}, $O(2)$ is
sufficient for any finite classification problem. Larger compact groups such as
$O(m)$ or $U(m)$ provide a richer space of solutions --- including curved ones ---
but do not increase the set of solvable problems for finite data.

A geometric analogue of the universal approximation theorem \cite{hornik1989} was proposed in earlier
drafts as a conjecture. We now prove it. This theorem addresses \textbf{Problem B}
(continuous boundary representation): given a continuous $\Gam$, how well can it
be approximated by an $O(2)$ $L$-harmonic interpolant?

The core density argument (Steps~1--2 below) is essentially the universal kernel
theory of Steinwart \cite{steinwart2001} and Micchelli--Xu--Zhang
\cite{micchelli2006}: density of the Green's function span in $C(M)$ for
$\nu > n/2$ is their $c$-universality result. We strengthen the hypothesis
slightly to $\nu > n/2 + 1$ to obtain $C^1$ control (Step~2), which is needed
for the level-set approximation (Step~4). The $O(2)$ section construction
(Steps~5--6) is the geometric packaging. The theorem should be understood as a
geometric restatement of universal kernel theory, not a new density result.

\begin{theorem}[Geometric universality]\label{thm:universal}
Let $(M, g)$ be a compact smooth Riemannian $n$-manifold without boundary, and let
$K = \Gr_g$ be the Green's function of $L = (\Delta_g + \kappa^2)^\nu$ for some
$\nu > n/2 + 1$ and $\kappa > 0$. For any continuous decision boundary
$\Gamma \subset M$ arising as a regular level set of some $\phi_0 \in C^\infty(M)$,
and for any $\varepsilon > 0$, there exist $N \in \mathbb{N}$, points
$x_1,\ldots,x_N \in M$, and coefficients $\alpha_1,\ldots,\alpha_N \in \mathbb{R}$
such that the $O(2)$ $L$-harmonic interpolant
$\phi_N(x) = \sum_{i=1}^N \alpha_i K(x, x_i)$
satisfies:
\begin{enumerate}[label=(\roman*), itemsep=2pt]
  \item $d_H\!\left(\phi_N^{-1}(\pi/2),\,\Gamma\right) < \varepsilon$
        \textnormal{(the zero level set approximates $\Gamma$)};
  \item $\mathrm{sign}(\cos\phi_N(x_i)) = y_i$ for every training point
        \textnormal{(correct sign conditions on all data)}.
\end{enumerate}
\end{theorem}

\begin{proof}
\textbf{Step 1} (\emph{$L^2$ density of $\{K(\cdot,x)\}$}). \;
Let $T_K\colon L^2(M,g) \to L^2(M,g)$ be the integral operator with kernel $K$.
Since $L$ is a positive-definite elliptic operator with $K = L^{-1}$ (as an integral
operator), $T_K$ is the inverse of $L$ on $L^2(M,g)$, hence bounded and injective.
Suppose $f \in L^2(M,g)$ satisfies $\langle f, K(\cdot,x)\rangle_{L^2} = 0$ for all
$x \in M$.  Then $(T_K f)(x) = 0$ for all $x$, giving $T_K f = 0$ in $L^2$, so
$f = 0$ by injectivity.  Therefore
\[
  \overline{\mathrm{span}\{K(\cdot,x) : x \in M\}}^{L^2(M,g)} \;=\; L^2(M,g).
\]

\textbf{Step 2} (\emph{RKHS identification and $C^1(M)$ density}). \;
By the Kimeldorf--Wahba correspondence \cite{kimeldorf1971some}, the reproducing
kernel Hilbert space of $K$ is $\mathcal{H}_K = H^\nu(M,g)$ (Sobolev space of
order $\nu$), with squared norm $\|f\|_{\mathcal{H}_K}^2 = \langle f, Lf\rangle_{L^2}$.
The reproducing property $f(x) = \langle f, K(\cdot,x)\rangle_{\mathcal{H}_K}$
implies that any $f \in \mathcal{H}_K$ orthogonal to all $K(\cdot,x)$ must vanish
identically; hence $\mathrm{span}\{K(\cdot,x) : x \in M\}$ is dense in
$\mathcal{H}_K$.  The Sobolev embedding theorem \cite{sobolev1938} gives a
continuous dense inclusion $H^\nu(M,g) \hookrightarrow C^1(M)$ for $\nu > n/2+1$
(one derivative stronger than the $C^0$ embedding, needed for Step~4).
Chaining these two densities:
\[
  \overline{\mathrm{span}\{K(\cdot,x) : x \in M\}}^{C^1(M)} \;=\; C^1(M).
\]

\textbf{Step 3} (\emph{Finite $C^1$ approximation of $\phi_0$}). \;
Let $\phi_0 \in C^\infty(M)$ satisfy $\phi_0^{-1}(\pi/2) = \Gamma$ with
$\nabla\phi_0 \neq 0$ on $\Gamma$.  By Step~2, for any $\delta > 0$ there exist
$N \in \mathbb{N}$, points $x_1,\ldots,x_N \in M$, and $\alpha_1,\ldots,\alpha_N
\in \mathbb{R}$ such that
\[
  \|\phi_N - \phi_0\|_{C^1(M)} \;<\; \delta,
  \qquad \phi_N(x) = \sum_{i=1}^N \alpha_i K(x, x_i).
\]

\textbf{Step 4} (\emph{Hausdorff approximation of $\Gamma$}). \;
Since $\nu > n/2+1$, Step~3 gives $C^1$ control: $\|\phi_N - \phi_0\|_{C^1} < \delta$.
In particular, for all $x$ near $\Gamma$:
\[
  |\nabla\phi_N(x) - \nabla\phi_0(x)| \;<\; \delta.
\]
Since $|\nabla\phi_0| \geq c_0 > 0$ on $\Gamma$ (regularity of the level set), for
$\delta < c_0/2$ we have $|\nabla\phi_N| > c_0/2 > 0$ near $\Gamma$. The implicit
function theorem therefore applies to $\phi_N$ near $\Gamma$, and the standard
quantitative version gives a constant $c > 0$ (depending only on $c_0 =
\inf_\Gamma|\nabla\phi_0|$ and the $C^1$ geometry of $\phi_0$) such that
\[
  d_H\!\bigl(\phi_N^{-1}(\pi/2),\;\phi_0^{-1}(\pi/2)\bigr)
  \;\leq\; c\,\delta.
\]
Choose $\delta = \varepsilon/c$ to obtain $d_H(\phi_N^{-1}(\pi/2), \Gamma) <
\varepsilon$.

\textbf{Step 5} (\emph{Flat $O(2)$ section}). \;
The section $s_N(x) = r(\cos\phi_N(x), \sin\phi_N(x))^\top$ is horizontal for the
connection $\theta = -d\phi_N \in \Omega^1(M, \mathfrak{so}(2))$.  Curvature:
$F = -d^2\phi_N\cdot J = 0$; the connection is flat.  Embedded as
$s_N \hookrightarrow (s_N, 0,\ldots,0)^\top \in \R^m$ for any $m \geq 2$, it is
a parallel section of $E = M \times \R^m$ with structure group $O(2) \subset O(m)$.
With readout $\varphi = (1,0,\ldots,0)$, the classifier is $r\cos\phi_N$.

\textbf{Step 6} (\emph{Sign conditions}). \;
Choose $\delta < \min_i |\phi_0(x_i) - \pi/2|$, which is positive since the
training data lies strictly off $\Gamma$.  Then $|\phi_N(x_i) - \phi_0(x_i)| < \delta$
implies $\phi_N(x_i)$ and $\phi_0(x_i)$ lie on the same side of $\pi/2$, so
$\mathrm{sign}(\cos\phi_N(x_i)) = \mathrm{sign}(\cos\phi_0(x_i)) = y_i$.
\end{proof}

\begin{remark}[Compactness and the worked examples]\label{rem:compactness}
Theorem~\ref{thm:universal} requires $(M,g)$ compact without boundary. This covers
the text-classification setting of \S\ref{sec:sphere} ($M = S^{n-1}$, compact) but
\emph{not} the worked examples of \S\ref{sec:svm}--\S\ref{sec:flower} and the
experiment of \S\ref{sec:experiment}, which all live on $M = \R^2$ (non-compact).
On non-compact $M$, the Matérn kernel is still well-defined and the capacitance
linear solve still produces an interpolant, but the $C^1(M)$ density argument of
Steps~1--2 requires modification: one needs either a compact support assumption,
decay conditions at infinity, or a weighted Sobolev space framework. The finite
interpolation problem (Theorem~\ref{thm:flat}) is unaffected by compactness ---
the partition-of-unity construction works on any smooth manifold. The universality
theorem should be understood as covering the compact setting; its extension to
$\R^n$ is an open problem.
\end{remark}

% ================================================================
\section{Data Symmetry, Structure Groups, and Efficiency}\label{sec:symmetry}
% ================================================================

\subsection{The symmetry group of the data}

Remark~\ref{rem:textbook} observed that $O(2)$ appears in textbook examples
because those examples are built using rotations. This raises a precise question:
for a given classification problem, what is the \emph{right} structure group?

We propose the following principle: the natural structure group for a classification
problem is the \textbf{symmetry group of the data distribution}. More precisely,
let $\mathcal{D}$ be the data distribution on $M$, and let $G_{\mathrm{data}}$ be
the group of transformations of $M$ that leave $\mathcal{D}$ invariant and commute
with the labelling. Then the bundle with structure group $G_{\mathrm{data}}$ is
the geometrically natural choice: the connection can exploit the data symmetry
directly, and the parallel section inherits the same symmetry.

For real classification problems the data symmetry is rarely $O(2)$. It may be:
\begin{itemize}[itemsep=2pt]
  \item A \textbf{discrete group} $\mathbb{Z}_k$ or $S_n$ (permutation symmetry
        in tabular data);
  \item A \textbf{product group} $G_1 \times G_2 \times \cdots$ when independent
        features have independent symmetries;
  \item A \textbf{continuous group} such as $SO(3)$ (3D point clouds, molecular
        data) or $SE(3)$ (rigid body transformations);
  \item \textbf{No symmetry at all} (generic tabular or text data), in which case
        the universal $O(2)$ solution is the only geometric option.
\end{itemize}

\subsection{Symmetry efficiency: the main observation}

The connection between the symmetry of the data and the structure group has a
concrete consequence for efficiency. Before stating the conjecture, we must be
precise about its scope and — crucially — its definition.

\paragraph{Scope.}
Theorem~\ref{thm:flat} shows that for any \emph{finite} dataset, fiber dimension
$m = 2$ suffices for any $G \supseteq O(2)$. So for finite data,
$m^*(G_{\mathrm{data}}) = m^*(G) = 2$ and the inequality is trivially $2 \leq 2$
with no content. The conjecture is therefore a statement about the
\textbf{continuous boundary approximation problem}: given a continuous decision
boundary $\Gam \subset M$ (invariant under $G_{\mathrm{data}}$) and a target
precision $\varepsilon > 0$, what is the minimum fiber dimension needed to
approximate $\Gam$ to within $\varepsilon$?

\paragraph{The equivariance constraint.}
Note first that $G$ appears here in the doubled sense of
Remark~\ref{rem:two-roles}: an action on the base $M$ together with a
representation $\rho$ on the fibre. This is a stronger structure than the
fibrewise structure group of \S\ref{sec:nn}, and it is what makes the
comparison below meaningful.
A key subtlety: without an additional constraint, a larger structure group
$G \supseteq G_{\mathrm{data}}$ trivially \emph{helps}, since any
$G_{\mathrm{data}}$-bundle is a $G$-bundle (by inclusion) and any
$G_{\mathrm{data}}$-equivariant section is a $G$-section. This would give
$m^*(G,\varepsilon) \leq m^*(G_{\mathrm{data}},\varepsilon)$ for free ---
the \emph{opposite} direction, with no content.

The conjecture becomes non-trivial only when we require sections to be
\textbf{$G$-equivariant}: the section must transform according to a specific
representation $\rho: G \to O(m)$ of the \emph{full} structure group $G$, not
merely exist as some section of a $G$-bundle. Formally:

\begin{definition}\label{def:mstar}
Let $G_{\mathrm{data}}$ act on $M$, and let $G \supseteq G_{\mathrm{data}}$
with a unitary representation $\rho: G \to O(m)$ on the fiber $\R^m$. A section
$s: M \to E$ is \emph{$G$-equivariant} if $s(g \cdot x) = \rho(g)\, s(x)$
for all $g \in G$ and $x \in M$. For a $G_{\mathrm{data}}$-invariant continuous
boundary $\Gam$, define
\[
  m^*(G,\, \varepsilon) \;=\; \min\bigl\{m \in \mathbb{N} \;:\;
  \exists\text{ a $G$-equivariant section }s\text{ of a rank-$m$ $G$-bundle with }
  d_H\bigl(\{\varphi(s)=0\},\,\Gam\bigr) < \varepsilon \bigr\}.
\]
\end{definition}

The equivariance constraint makes $m^*(G,\varepsilon)$ genuinely sensitive to the
choice of $G$: a $G$-equivariant section must respect \emph{all} symmetries of $G$,
not just those of $G_{\mathrm{data}}$. When $G \supsetneq G_{\mathrm{data}}$, the
space of $G$-equivariant candidates is a proper sub-space of the
$G_{\mathrm{data}}$-equivariant candidates, potentially requiring more fiber
dimensions to achieve the same approximation quality.

\begin{conjecture}[Symmetry efficiency]\label{conj:symmetry}
Let $G_{\mathrm{data}} \subseteq G$ act on $M$, and let $\Gam$ be a continuous
boundary with symmetry group containing $G_{\mathrm{data}}$. With $m^*(G,\varepsilon)$
as in Definition~\ref{def:mstar}:
\[
  m^*(G_{\mathrm{data}},\,\varepsilon) \;\leq\; m^*(G,\,\varepsilon)
  \quad \text{for all } \varepsilon > 0.
\]
\end{conjecture}

The conjecture says: the matched group $G_{\mathrm{data}}$ is the most efficient
choice because its equivariant sections span the largest function space consistent
with the data symmetry. A mismatched $G \supsetneq G_{\mathrm{data}}$ imposes
additional equivariance requirements not present in $\Gam$, reducing the space of
candidates and potentially forcing a larger fiber dimension to maintain the same
approximation quality. The representation-theoretic content is: the irreducible
representations of $G_{\mathrm{data}}$ that appear in the decomposition of a
$G_{\mathrm{data}}$-equivariant approximating function need not appear as
irreps of $G$ — the larger group may require more channels to represent the
same geometric content.

\subsection{Connection to equivariant networks}

Conjecture~\ref{conj:symmetry} gives the geometric content of equivariant network
theory: matching the architecture symmetry group $G$ to the data symmetry
$G_{\mathrm{data}}$ reduces the required fiber dimension, equivalently the
number of neurons in a one-hidden-layer network. The geometric deep learning
programme \cite{bronstein2021} is therefore a special case of this conjecture, with the connection
framework giving the empirical observation a precise geometric formulation.

% ================================================================
\section{Persistent Homology as a Guide to the Structure Group}\label{sec:tda}
% ================================================================

The question of how to select the structure group $G$ from data admits a partial
answer via topological data analysis \cite{edelsbrunner2010}. Given a labelled dataset $\mathcal{D}$, apply
a Vietoris--Rips or \v{C}ech filtration to the point cloud $\{x_i\}$ and read off
the Betti numbers $\beta_0, \beta_1, \beta_2, \ldots$ at the dominant persistence
scale.

\begin{remark}[What the Betti numbers do and do not determine]\label{rem:tda-shape}
It is tempting to write a product such as
$O(2)^{\beta_1}\times SO(2)^{\beta_2}$, one factor per independent cycle. That is
not how bundles work. For a \emph{single} structure group $G$, the flat
$G$-bundles over a space $X$ are classified by $\mathrm{Hom}(\pi_1(X), G)/G$ ---
all $\beta_1$ loops are accounted for by one group, since a homomorphism assigns
a holonomy element to every cycle simultaneously. Likewise all $\beta_2$ classes
are carried by a single $SO(2)$-bundle through its Euler class in $H^2(X;\mathbb{Z})$;
one bundle already encodes the whole cohomology class. The Betti numbers
constrain the \emph{size of the invariant}, i.e.\ the rank of
$\mathrm{Hom}(H_1, \mathbb{Z}_2)$ or of $H^2$, not the number of group factors.
\end{remark}

Accordingly the correct reading is: $\beta_1$ bounds the rank of the holonomy
data a single $G$ must carry, and $\beta_2$ bounds the rank of the available
Euler/Chern classes. The pipeline is a guide to \emph{how much} topological
structure $G$ must support, not to a factorisation of $G$.

The pipeline is principled but not complete. Persistent homology captures global
topological features but misses local geometric symmetry (e.g.\ $SO(3)$
equivariance in 3D point clouds), non-abelian $\pi_1$ (Betti numbers give only the
abelianisation $\pi_1^{\mathrm{ab}} = H_1$), and ambient domain symmetries such
as translation or permutation invariance. Known domain symmetries should be
added separately.

A category remark is needed: the Betti numbers $\beta_k$ are computed from the
\emph{point cloud} $\{x_i\}$, not from the base manifold $M$. If $M = \R^n$
(contractible), all bundles over $M$ are trivial regardless of $\beta_1$.
The intended object is the topology of the \emph{support of the data
distribution} --- the sub-manifold or subset of $\R^n$ near which the data
concentrates --- which may have non-trivial loops and voids even when $M$ itself
does not. The proposal below is a heuristic for reading that support topology
from the point cloud.

\begin{heuristic}[Topological structure group]\label{prop:tda-lower}
Let $X$ denote the support of the data distribution, with Betti numbers
estimated from the point cloud. A structure group $G$ intended to represent the
topology of $X$ should be large enough that
$\mathrm{Hom}(H_1(X), G)$ and $H^2(X; \pi_1(G))$ are non-trivial whenever the
corresponding Betti numbers are; in particular $G$ abelian of rank $1$
(e.g.\ $O(2)$) already suffices to carry $\mathbb{Z}_2$-holonomy around every
$1$-cycle and an Euler class on every $2$-cycle.

This is a heuristic proposal, not a theorem. The predicate ``large enough to
represent the topology'' is not defined with sufficient precision to constitute
a mathematical statement, and the connection between point-cloud Betti numbers
and bundle structure over $M$ requires further development.
\end{heuristic}

Even when this reading captures only the topological component of the optimal
$G^*$, it provides a principled, computable starting point for structure group
selection --- better than arbitrary choice of activation function. Making the
heuristic precise is an open problem (Q2 in \S\ref{sec:open}).

\medskip
\noindent\textit{Arc 2 ends here.} With the structure group identified,
\S\ref{sec:yang-mills} (Arc~3) formulates the exact geometric BVP that was
illustrated by the examples of Arc~1, and derives the $L$-harmonic interpolant as
its closed-form solution.

% ================================================================
\section{Classification as a Boundary Value Problem}\label{sec:yang-mills}
% ================================================================

\subsection{Classification as a geometric boundary value problem}

The fundamental shift proposed in this paper is the replacement of optimisation
over parameters with a geometric existence problem: find a connection whose
parallel section satisfies sign conditions at the data points. This section
addresses \textbf{Problem A} (finite labeled data): the data enters as Dirichlet
conditions at finitely many points, and the solution interpolates between them.
We now sharpen this: the data should enter not as terms in a functional to be
minimised, but as \textbf{Dirichlet boundary conditions} on the connection.

We formulate the problem using the Yang-Mills equations $D^*F = 0$ as a unifying
language, though the honest scope of this language must be stated at the outset.
In the solved abelian case (abelian $G$, so $[F,\cdot] = 0$), $D^*F = 0$ reduces to the scalar
Laplace equation $\Delta_g f = 0$ --- classical potential theory. In the flat
$O(2)$ case, $F = 0$ identically and the Yang-Mills condition is trivially
satisfied; the relevant equation is the Dirichlet problem for the angle function
$\phi$. Genuine non-abelian Yang-Mills theory, where $D^*F = 0$ is a non-trivial
PDE on a connection with $F \neq 0$, arises only for $G = O(m)$ with $m \geq 3$
--- which is the open problem of Q1 in \S\ref{sec:open}. The Yang-Mills language
is adopted because it correctly identifies the structural type of each BVP and
gives the right Euler-Lagrange equation in the non-abelian case; it is not a
claim that new Yang-Mills theory is being done.

Given a bundle $E \to M$ with structure group $G$, Riemannian metric $g$ on
$M$ (which determines the Laplace--Beltrami operator $\Delta_g$ and hence the
Green's function $\Gr_g$ and kernel $K = \Gr_g$), bundle metric on $E$, and
labeled dataset $\mathcal{D}$, the
classification problem for separable data is the following \emph{geometric
boundary value problem}: find the connection $\Sig$ that

\begin{enumerate}[label=(\roman*), itemsep=3pt]
  \item satisfies the \textbf{Yang-Mills equations} on the complement of the data:
  \begin{equation}\label{eq:YM-BVP}
    D^*_\Sig F_\Sig \;=\; 0 \qquad \text{on } M \setminus \{x_1,\ldots,x_N\},
  \end{equation}
  \item satisfies \textbf{Dirichlet conditions} at each data point:
  \begin{equation}\label{eq:dirichlet}
    \varphi\bigl(s_\Sig(x_i)\bigr) \;=\; y_i \cdot r, \qquad i = 1,\ldots,N,
  \end{equation}
  \item is selected among the connections satisfying (i) and (ii) by the
        two-stage rule of Definition~\ref{def:solution}: first minimise the
        Yang--Mills energy $\int_M\norm{F_\Sig}^2\,\mathrm{dvol}_g$, then among
        those minimisers minimise the potential energy
        $\langle u, Lu\rangle = \norm{u}^2_{H^\nu}$
        (see Remarks~\ref{rem:two-stage} and~\ref{rem:not-dirichlet}).
\end{enumerate}

There is \textbf{no loss function, no gradient descent}. The data
enters purely as prescribed values of the section at finitely many points.

\begin{remark}[Geometric parameters vs optimisation parameters]
The framework eliminates two things: the loss function (data enters as hard
Dirichlet constraints, not a penalty) and iterative gradient descent (the solution
is a single linear solve). It does not eliminate parameters. The geometric
parameters --- length scale $\ell$ (determining the Riemannian metric $g$),
Matérn order $\nu$ (determining the RKHS regularity), Tikhonov regularisation
$\varepsilon$ (required for well-posedness, see below), and section scale $r$
(fiber metric) --- play a role analogous to $\lambda$ in the penalised formulation.
What changes is their interpretation: each is a property of the geometry of the
data manifold or bundle, not an arbitrary knob tuned against a training loss. The
length scale $\ell$ is the choice of Riemannian metric; $\nu$ is the order of the
elliptic operator; $\varepsilon$ is a conductor radius in the electrostatic
analogy. The reframing is genuine --- the parameters have precise geometric meaning
--- but parameter-free it is not.
\end{remark}

\begin{remark}[Why Dirichlet, not penalisation]
The penalised functional $\int\norm{F}^2 + \lambda\sum_i\ell(y_i,\varphi(s(x_i)))$
is a soft relaxation of the Dirichlet problem: as $\lambda \to \infty$ it converges
to the hard constraint \eqref{eq:dirichlet}. Writing the functional with finite
$\lambda$ carries the implicit assumption that approximate satisfaction of the data
conditions is acceptable --- which is the machine learning mindset. The geometric
mindset: for separable data, the constraints can and should be satisfied exactly.
The $\lambda$ functional appears naturally only for non-separable data
(\S\ref{subsec:nonsep}) or as a degenerate limit recovering backpropagation
(\S\ref{subsec:backprop}).
\end{remark}

\subsection{Exact solutions: the abelian case and $L$-harmonic interpolation}
\label{subsec:abelian}

For $G = \mathbb{R}^+$ (abelian, real line bundle), the Yang-Mills equations \eqref{eq:YM-BVP} reduce,
via the connection 1-form $A = -d\log|f|$ of \S\ref{sec:svm}, to a
\textbf{scalar elliptic equation} for the classifier $f: M \to \R$ away from
the data. We work throughout with the regularised operator
$L = (\Delta_g + \kappa^2)^\nu$ for $\nu > n/2$ and $\kappa > 0$, whose
Green's function $\Gr_L$ is positive definite and finite on the diagonal
(Remark~\ref{rem:regularisation}). The interpolant satisfies
\[
  L f \;=\; \sum_i \alpha_i y_i\, \delta_{x_i}
  \qquad\text{so}\qquad
  Lf = 0 \ \text{ on } M \setminus \{x_1,\dots,x_N\} ,
\]
i.e.\ $f$ is \textbf{$L$-harmonic} away from the data. It is not harmonic:
$\Delta_g f \neq 0$ for $\nu \neq 1$.

Among all $f$ satisfying the Dirichlet conditions, the one minimising the
\textbf{energy}
\[
  \mathcal{E}_L[f] \;=\; \langle f, Lf\rangle_{L^2}
  \;=\; \int_M \norm{(\Delta_g+\kappa^2)^{\nu/2} f}^2\,\mathrm{dvol}_g
  \;=\; \norm{f}^2_{\mathcal{H}_K}, \qquad K = \Gr_L,
\]
--- the squared $H^\nu$ norm, \emph{not} the Dirichlet energy --- is:

\begin{equation}\label{eq:harmonic}
  \boxed{f(x) \;=\; \sum_i \alpha_i\,y_i\,\Gr_L(x, x_i).}
\end{equation}

\begin{remark}[Why not the Dirichlet energy]\label{rem:not-dirichlet}
The Dirichlet energy $\int_M|\nabla f|^2$ is the case $\nu = 1$, $\kappa = 0$ of
$\mathcal{E}_L$, and for point data it is the wrong functional in a strong sense.
Point evaluation is bounded on $H^\nu(M)$ only for $\nu > n/2$; for $n \geq 2$
this excludes $\nu = 1$. Points have zero $H^1$-capacity, so
$\inf\{\int|\nabla f|^2 : f(x_i) = y_i r\} = 0$ and is not attained --- the
constrained Dirichlet problem with point data has no solution. Consistently, the
$n=2$ Green's function $-\tfrac{1}{2\pi}\log|x-y|$ has infinite Dirichlet energy.
The condition $\nu > n/2$ that makes $\Gr_L$ a reproducing kernel is exactly the
condition that makes the interpolation problem well posed, and it is incompatible
with the Dirichlet energy in dimension $n \geq 2$.
\end{remark}

Note that the connection curvature $F = dA = d(-d\log|f|) = 0$ vanishes
identically away from $\Gam$ for \emph{any} smooth $f$ --- so it is
$\mathcal{E}_L$, the energy of the potential, and not the curvature of the
induced connection, that selects the interpolant. The coefficients $\alpha_i$
are determined by the \textbf{capacitance matrix equation}: substituting
\eqref{eq:harmonic} into the Dirichlet conditions \eqref{eq:dirichlet} gives a
\emph{linear system}:
\begin{equation}\label{eq:capacitance}
  \mathbf{K}\,\boldsymbol{\alpha} \;=\; \mathbf{v},
  \qquad K_{ij} = \Gr_L(x_i, x_j),\quad v_i = y_i \cdot r,
\end{equation}
solved by $\boldsymbol{\alpha} = \mathbf{K}^{-1}\mathbf{v}$. This is a single
linear system inversion --- no iterations, no gradient steps. The classifier is
the \textbf{$L$-harmonic interpolant} of the prescribed values at the data
points, extended to all of $M$ by the minimum-$\mathcal{E}_L$ condition.

\subsection{The kernel is the Green's function}\label{subsec:kernel}

The $L$-harmonic interpolant \eqref{eq:harmonic} has the same structure as an RKHS
interpolant $f(x) = \sum_i \alpha_i y_i K(x_i,x)$. The identification is exact:

\begin{theorem}[RKHS interpolation is the abelian $L$-harmonic interpolant]%
\label{thm:svm-ym}
Let $L_K$ be a self-adjoint positive-definite elliptic operator on $(M,g)$
whose Green's function $K(x,y) = \Gr_{L_K}(x,y)$ is \emph{finite on the
diagonal} (e.g.\ $L_K = (\Delta_g+\kappa^2)^\nu$ with $\nu > n/2$,
$\kappa > 0$). Then:
\begin{enumerate}[label=(\roman*), itemsep=3pt]
  \item The kernel $K$ is the reproducing kernel of the Hilbert space
        $\mathcal{H}_K = \mathrm{range}(L_K^{-1/2})$ with norm
        $\norm{f}^2_{\mathcal{H}_K} = \langle f, L_K f\rangle$.
  \item The RKHS interpolant $f(x) = \sum_i\alpha_i y_i K(x_i,x)$ with
        $\boldsymbol{\alpha} = \mathbf{K}^{-1}\mathbf{v}$ is the unique
        minimum-$\mathcal{H}_K$-norm function satisfying the Dirichlet conditions
        $f(x_i) = y_i \cdot r$. It solves the abelian $L$-harmonic interpolation
        problem \eqref{eq:harmonic}--\eqref{eq:capacitance} with operator $L_K$.
  \item Write $\Gr_{L}$ for the Green's function of $L = (\Delta_g+\kappa^2)^{\nu}$.
        Then $\mathcal{H}_{\Gr_L} = H^{\nu}(M,g)$ (Sobolev space of order $\nu$)
        and the solution minimises $\int\norm{(\Delta_g+\kappa^2)^{\nu/2}f}^2$.
        In $\R^n$ this Green's function is the Mat\'ern covariance of
        \emph{smoothness} $\nu_M = \nu - n/2$; the operator exponent $\nu$ and
        the Mat\'ern smoothness parameter $\nu_M$ differ by $n/2$ and should
        not be conflated.
  \item Different kernels correspond to different operators $L_K$, equivalently
        different Riemannian metrics and regularisation orders on $M$. The choice
        of kernel is the choice of geometric structure on the data manifold.
\end{enumerate}
\end{theorem}

\begin{corollary}\label{cor:svm-ym}
The RKHS interpolant solves the electrostatic capacitance problem: each data
point $(x_i, y_i)$ acts as a conductor at prescribed potential $y_i \cdot r$,
and $f$ is the resulting Green's function field. The weights
$\boldsymbol{\alpha} = \mathbf{K}^{-1}\mathbf{v}$ are the capacitance
coefficients. No optimisation --- classical potential theory.

The hard-margin SVM is a related but distinct object: it minimises
$\|f\|^2_{\mathcal{H}_K}$ subject to \emph{inequality} constraints
$y_i f(x_i) \geq 1$ (the margin), so most $\alpha_i$ are zero. The SVM is
the RKHS interpolant \emph{restricted to its own support vectors} --- the active
constraints $y_i f(x_i) = 1$ --- whose identity is determined endogenously by
the quadratic programme. The two estimators coincide if and only if all training
points happen to be support vectors.
\end{corollary}

\begin{remark}[Relationship to known results]\label{rem:known}
The representer theorem for interpolation \cite{kimeldorf1971some} establishes
that the minimum-$\mathcal{H}_K$-norm function satisfying $f(x_i) = y_i r$ has
the form $f = \sum_i \alpha_i K(x_i,\cdot)$ with $\boldsymbol{\alpha} =
\mathbf{K}^{-1}\mathbf{v}$ --- this is the mathematical content of
Theorem~\ref{thm:svm-ym}(i)--(ii). The identification of Mat\'ern kernels as
Green's functions of $(\Delta_g+\kappa^2)^\nu$ is the SPDE correspondence of
Lindgren, Rue, and Lindstr\"om \cite{lindgren2011spde}. The SVM dual solution
being supported on active constraints is standard KKT theory \cite{scholkopf2002}.
What Theorem~\ref{thm:svm-ym} and Corollary~\ref{cor:svm-ym} add is the
geometric interpretation: the linear system is the capacitance equation of
classical potential theory; the solution is the harmonic potential of point
sources at $\{x_i\}$; and the decision boundary is the zero equipotential.
The \emph{dictionary} between the RKHS interpolant and the Dirichlet BVP is the
contribution, not the analytic correspondence itself.
\end{remark}

\begin{remark}
The Gaussian (RBF) kernel corresponds to the limit $\nu\to\infty$ of the Mat\'ern
family --- an infinitely smooth Riemannian structure --- and has no finite-order
Green's function representation. Every other common kernel (Mat\'ern, polynomial,
Laplacian) has an exact Green's function interpretation. Kernels whose RKHS is
dense in $C(M)$ --- equivalently, whose Green's function span is total --- are
called \emph{universal kernels} \cite{steinwart2001}; these are precisely the
kernels for which the geometric universality theorem (Theorem~\ref{thm:universal})
applies in its strongest form.
\end{remark}

\subsection{Exact solutions: prescribed angle values for $O(2)$ on $\R^2$}
\label{subsec:O2}

For $G = O(2)$ on $M = \R^2$, the section angle $\phi$ satisfies
$\Delta\phi = 0$ on $\R^2\setminus\{x_i\}$ with prescribed values:
\[
  \phi(x_i) \;=\; \begin{cases} 0 & y_i = +1 \\ \pi & y_i = -1. \end{cases}
\]
The Green's function of $\Delta$ on $\R^2$ is $\Gr(x,y) = -\frac{1}{2\pi}\log|x-y|$,
giving the formal interpolant:
\begin{equation}\label{eq:phi-harmonic}
  \phi(x) \;=\; -\sum_i c_i\log|x - x_i| + \phi_0,
  \qquad \mathbf{G}\,\mathbf{c} = \mathbf{u},
\end{equation}
where $G_{ij} = -\frac{1}{2\pi}\log|x_i-x_j|$ and
$u_i = \phi(x_i) \in \{0,\pi\}$.

\begin{remark}[Regularisation requirement]\label{rem:regularisation}
The diagonal entries $G_{ii} = -\frac{1}{2\pi}\log 0 = +\infty$ make the
capacitance matrix $\mathbf{G}$ ill-defined for idealised point sources. In the
electrostatic analogy, point charges have infinite self-energy; the capacitance
problem is well-posed only for conductors of finite radius $\delta > 0$.
Algebraically, replacing $\Delta$ by $L = (\Delta + \kappa^2)^\nu$ gives a
Green's function $\Gr_L(x,y)$ that is finite on the diagonal and makes
$\mathbf{G}$ positive definite. The parameter $\kappa$ (equivalently the length
scale $\ell$ in the Matérn kernel) is the geometric regularisation parameter ---
the conductor radius in the electrostatic picture. This is honest: the method
requires a regularisation parameter; it gives it geometric meaning rather than
tuning it against a loss.
\end{remark}

This is the \emph{stream function of a 2D point vortex gas}: each data point is
a logarithmic vortex of strength $c_i$ determined by its prescribed angle value.
The decision boundary $\Gam = \{\phi = \pi/2\}$ is the equipotential at height
$\pi/2$. For the regularised operator $L = (\Delta+\kappa^2)^\nu$ the solution
$\phi(x) = \sum_i c_i \Gr_L(x,x_i) + \phi_0$ is a superposition of Green's
potentials at fixed positions --- classical 2D potential theory, no numerical
iterations. (The positions are fixed, so no vortex dynamics arise.) The
unregularised expression with $\log$ kernels is formal only: see
Remark~\ref{rem:regularisation}.

\begin{remark}[XOR in the regularised vortex picture]\label{rem:xor-vortex}
For XOR the prescribed values are $\phi(0,0) = \phi(1,1) = \pi$ and
$\phi(0,1) = \phi(1,0) = 0$. These cannot be imposed on the unregularised
potential $\phi(x) = -\sum_i c_i\log|x-x_i| + \phi_0$, which diverges at every
data point; the matrix $\mathbf{G}$ has $G_{ii} = +\infty$ and $\mathbf{c}$ is
undefined (Remark~\ref{rem:regularisation}). With the regularised kernel
$\Gr_L$, $L = (\Delta+\kappa^2)^\nu$, the capacitance system is well posed and
yields a decision boundary
\[
  \Gam \;=\; \Bigl\{x : \textstyle\sum_i c_i\,\Gr_L(x,x_i) = \tfrac{\pi}{2} - \phi_0\Bigr\},
\]
a smooth closed curve separating the two diagonal pairs. This differs from the
parallel-stripe construction $|x^1-x^2| = 1/\sqrt{2}$ of
Theorem~\ref{thm:flat}, and the two are not competing answers to one question:
Theorem~\ref{thm:flat} is an \emph{existence} statement, exhibiting one
admissible pair with no claim of canonicity, whereas the construction here is
the \emph{selected} solution of Definition~\ref{def:solution}. Both are
admissible; only the second minimises $\mathcal{E}_L$ at stage~(ii). In the
limit $\kappa \to 0$ the boundary approaches an Apollonius-type locus,
but that limit is formal --- the point-source problem itself has no solution.
\end{remark}

\begin{remark}[Different tuples, different solutions]\label{rem:tuple-dependence}
The abelian construction of \S\ref{subsec:abelian} minimises $\mathcal{E}_L$ over
the classifier $f$ itself; the $O(2)$ construction of this section minimises it
over the fibre angle $\phi$, with $\varphi = (1,0)$ so that the classifier is
$r\cos\phi$. These give different decision boundaries, but they are solutions to
different problems: the tuples
$(M,g,E,\R^+,\mathrm{id},\mathcal{D})$ and $(M,g,E,O(2),(1,0),\mathcal{D})$
differ in both structure group and readout. Definition~\ref{def:solution} selects
a unique solution \emph{given a tuple}; it does not privilege one tuple over
another. Choosing the tuple is the modelling step, and \S\ref{sec:tda} discusses
how the data topology can inform it.
\end{remark}

\subsection{The non-separable case: geometric relaxation}
\label{subsec:nonsep}

When no connection satisfying the geometric BVP admits a parallel section
fulfilling all Dirichlet conditions exactly
(non-separable data), we relax the equality constraints to \emph{inequality
constraints}: require only that the section values lie in the correct
open half-fiber:
\[
  \varphi\bigl(s_\Sig(x_i)\bigr) \cdot y_i \;>\; 0 \qquad \forall\,i.
\]
The two-stage energy-minimising connection satisfying all inequalities is the constrained
geometric BVP. The geometric ``slack'' at data point $x_i$ is the signed
distance $\varphi(s_\Sig(x_i)) \cdot y_i / r$ --- how far inside the correct
half-fiber the section value lies. The SVM margin is the minimum of these
geometric slacks over the data.

\begin{remark}[The $\lambda$ functional as soft relaxation]
The penalised functional $\mathcal{F}[\Sig] = \int\norm{F}^2 + \lambda\sum_i
\ell(y_i,\varphi(s(x_i)))$ is the soft-margin (Lagrangian) relaxation of the
inequality-constrained problem, with $\lambda$ the inverse penalty for constraint
violation. It is a useful computational approximation when the constraints are
difficult to enforce directly, but it is not the primary formulation: it
introduces a free parameter $\lambda$ with no geometric meaning.
\end{remark}

\subsection{Gradient descent as degenerate Yang-Mills}
\label{subsec:backprop}

When $M$ is contractible and $F_\Sig = 0$ (flat connection), the Yang-Mills
equations \eqref{eq:YM-BVP} are satisfied trivially. The problem reduces to
finding a flat connection satisfying the sign conditions --- Theorem~\ref{thm:flat}
provides one. If instead of the exact geometric solution we seek to approximate it
by gradient flow on the space of flat connections:
\[
  \frac{\partial\Sig}{\partial t} \;=\; -\frac{\partial}{\partial\Sig}
  \sum_i \ell\bigl(y_i,\,\varphi(s_\Sig(x_i))\bigr),
\]
this is gradient descent on the data loss. If one could parametrise the space of
flat connections by the network weights, $\Sig \leftrightarrow (W_1, W_2)$, the
flow would be backpropagation.

\begin{remark}[The identification is structural, not constructive]
\label{rem:backprop-gap}
We have not constructed the correspondence $\Sig \leftrightarrow (W_1,W_2)$.
The obstruction is concrete: the linear network $x \mapsto W_2W_1x$ vanishes on
a subspace, so it does not define a nowhere-zero section and therefore induces
no connection by the construction of \S\ref{sec:curvature}. What holds is a
structural analogy --- both are flows on a space of flat objects driven by a
data loss, both converge asymptotically rather than solving a boundary value
problem --- and this is the sense in which we call backpropagation the
flat-connection limit throughout. Making the correspondence explicit would
require either an activation (so the section stays away from zero) or a
different parametrisation of the flat connections; we do not do so here.
\end{remark}

\begin{remark}[Backpropagation as the flat-connection limit]
Backpropagation is an iterative numerical method for the flat-connection
($F=0$) regime. It is \emph{degenerate} in the precise technical sense: it
restricts to flat connections, discarding curvature information, and replaces
the direct geometric solution with a gradient flow converging asymptotically.
The geometric BVP gives the solution in a single step, at the cost of
choosing a kernel (length scale $\ell$, order $\nu$). Neither approach is
universally superior: the iterative method is flexible and parameter-free
in a different sense; the direct solve is exact and fast for separable data.
\end{remark}

\subsection{The exact solution landscape}

\begin{center}
\renewcommand{\arraystretch}{1.3}
\begin{tabular}{@{}llll@{}}
\toprule
\textbf{Case} & \textbf{$G$} & \textbf{Solution method} & \textbf{Key result} \\
\midrule
Abelian / RKHS & $\mathbb{R}^+$ & Harmonic interpolation / capacitance
  & $\boldsymbol{\alpha} = \mathbf{K}^{-1}\mathbf{v}$ (linear solve) \\
$O(2)$ on $\R^2$ & $O(2)$ & 2D vortex gas / complex analysis
  & Apollonius-type locus \\
$S^2$, $e=k$ & $SO(2)$ & Dirac monopole + capacitance
  & Forced curvature, $k$ zeros \\
Non-abelian & $O(m)$ & Non-abelian harmonic maps
  & ADHM / Nahm transform \\
Non-separable & any & Inequality-constrained geometric BVP
  & Geometric margin maximisation \\
Backpropagation & any & Degenerate flat gradient flow
  & Iterative approximation only \\
\bottomrule
\end{tabular}
\end{center}

\medskip
In every well-posed case, the solution is found by \textbf{solving a linear or
geometric PDE system exactly} --- never by iterative parameter optimisation. The
framework makes the geometry explicit and the algorithm direct.

% ================================================================
\section{Numerical Experiment: Geometric vs.\ Neural-Network Solution}%
\label{sec:experiment}
% ================================================================

To illustrate the theoretical framework concretely, we compare the geometric
solution (Theorem~\ref{thm:svm-ym}) with a standard neural network on the same
problem. Since the geometric solution is an interpolant (no optimisation loop),
we report \emph{fit accuracy} on the training set (tautologically 100\% for any
interpolant) and \emph{test accuracy} on a held-out set as the meaningful
evaluation. For the neural network, both training and test accuracy are reported.

\subsection{Setup}

We generate $n_{\mathrm{train}} = 60$ labelled training points and
$n_{\mathrm{test}} = 300$ held-out test points from a two-moons distribution
(thirty points per class, Gaussian noise $\sigma = 0.13$).
The classification problem has tuple
\[
  \bigl(M = \R^2,\; g = g_{\mathrm{Eucl}},\; E = M\times\R,\;
        G = \mathbb{R}^+,\; \varphi = \mathrm{id},\; \mathcal{D}_{\mathrm{moons}}\bigr),
\]
where $g_{\mathrm{Eucl}} = (dx^1)^2+(dx^2)^2$ is the canonical flat metric on
$\R^2$, and its regularised Green's function is the Mat\'ern kernel.

\textbf{Geometric solution.}  We solve the abelian $L$-harmonic interpolation
problem (Theorem~\ref{thm:svm-ym}) on the training set.  The kernel is the
Mat\'ern-$3/2$ Green's function:
\[
  K(x,y) \;=\;
  \Bigl(1 + \tfrac{\sqrt{3}\,\|x-y\|}{\ell}\Bigr)
  \exp\!\Bigl(-\tfrac{\sqrt{3}\,\|x-y\|}{\ell}\Bigr),
  \qquad \ell = 0.55,
\]
which, by the correspondence $\nu = \nu_M + n/2$ of
Theorem~\ref{thm:svm-ym}(iii), is the Green's function of
$(\Delta + \kappa^2)^{5/2}$ on $\R^2$ (smoothness $\nu_M = 3/2$, $n = 2$).  The weights $\boldsymbol{\alpha} =
\mathbf{K}^{-1}\mathbf{v}$ are computed by a single $60\!\times\!60$ linear
system solve: no iterations, no gradient descent. A Tikhonov regularisation
$\varepsilon I$ ($\varepsilon = 10^{-6}$) is added to $\mathbf{K}$; this is
the geometric regularisation parameter (Remark~\ref{rem:regularisation}).

\textbf{Neural network.}  We train a \textbf{one-hidden-layer} network
$2 \to 64 \to 1$ with \texttt{tanh} activation (Adam optimiser,
learning rate $10^{-3}$, up to $5{,}000$ iterations) on the same training
set. This architecture matches the theoretical object $f = W_2\,\sigma(W_1 x)$
analysed in \S\ref{sec:nn}, giving a fair comparison: both models produce
classifiers of the same structural form, found by different methods.

\subsection{Results}

\begin{figure}[ht]
\centering
\includegraphics[width=\linewidth]{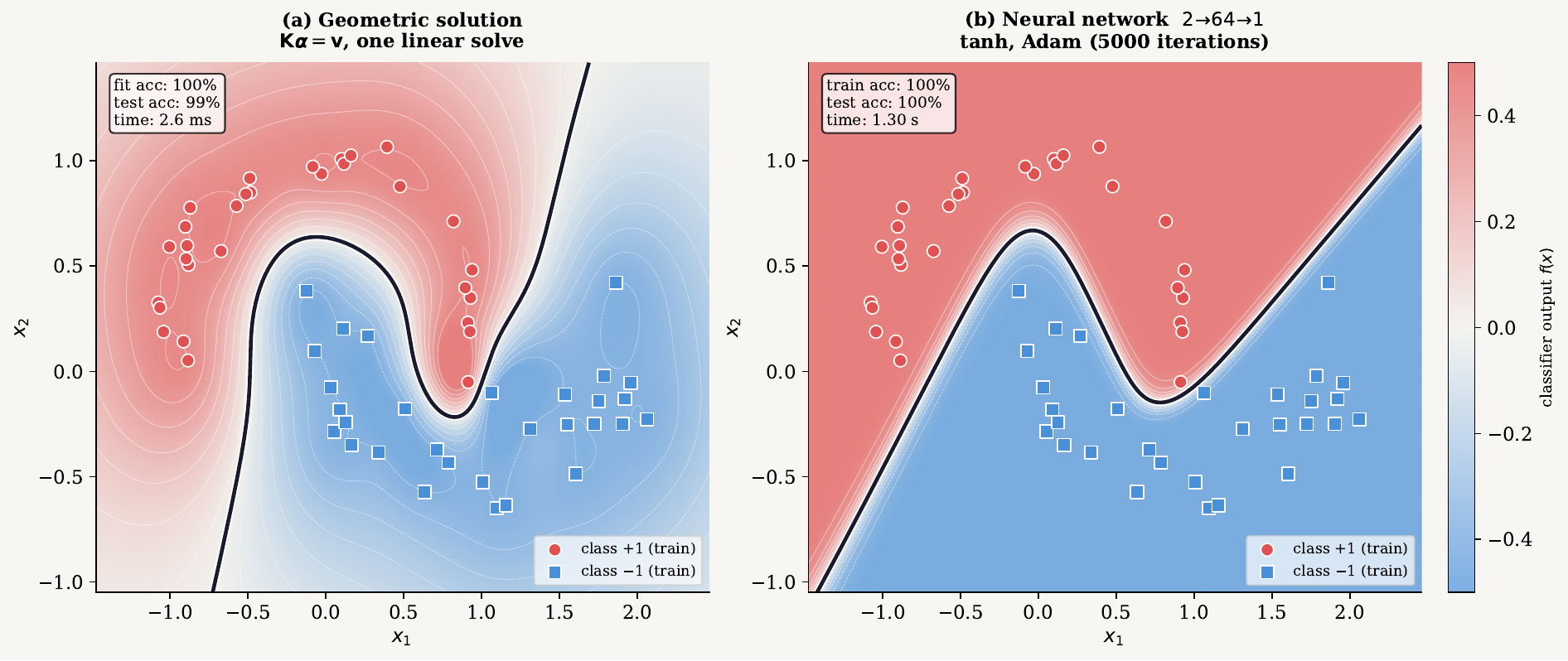}
\caption{Geometric solution \emph{(a)} vs.\ one-hidden-layer neural network
\emph{(b)} on the two-moons dataset ($n_{\mathrm{train}} = 60$,
$n_{\mathrm{test}} = 300$, Mat\'ern-$3/2$ kernel, $\ell = 0.55$). Training
points shown as markers; background field is $f(x)$ on a dense grid; dark
curve is the decision boundary.
\emph{(a)} Geometric solution: single $60\times 60$ linear solve, $2.6\,\mathrm{ms}$.
Fit acc.\ (training): $100\%$ by construction; test acc.: $99\%$.
\emph{(b)} One-hidden-layer network $2\to 64\to 1$, tanh: $5{,}000$ Adam
gradient steps, $1.30\,\mathrm{s}$ (gradient norm not converged at
$\mathrm{tol}=10^{-9}$; accuracy fully converged).
Train acc.: $100\%$; test acc.: $100\%$.}
\label{fig:geo-vs-nn}
\end{figure}

Both models achieve high test accuracy on this dataset; the comparison
demonstrates the paper's central claim about \emph{method}. The geometric
solution is an exact interpolant obtained from a single $60\!\times\!60$
linear system solve in $2.6\,\mathrm{ms}$. The one-hidden-layer network
requires $5{,}000$ gradient steps ($1.30\,\mathrm{s}$, a ${\approx}500\times$
slowdown) with the gradient-norm tolerance not reached at $\mathrm{tol}=10^{-9}$
--- though accuracy itself is fully converged. The experiment illustrates the
flat-connection limit of \S\ref{subsec:backprop}: gradient descent iteratively
approximates the same geometric problem that the linear solve resolves directly.

The decision boundaries differ qualitatively. The geometric boundary (panel
a) is the zero equipotential of the Green's function field of $L$: the
electrostatic picture of Corollary~\ref{cor:svm-ym}, where each training
point acts as a conductor at prescribed potential $\pm r$ and $\nabla f
\perp \Gamma$ at every boundary point. The equipotential lines make the
full field structure visible. The network boundary (panel b) is the emergent
zero of an iteratively optimised composite function, with no such natural
geometric structure.

The background colour field (Figure~\ref{fig:geo-vs-nn}, left) displays the
$L$-harmonic interpolant $f(x) = \sum_i \alpha_i K(x, x_i)$ over $\R^2$.

\subsection{Implementation note}

The geometric solution requires only \texttt{numpy} and \texttt{scipy}: assemble
the $N\times N$ kernel matrix $\mathbf{K}$ (Mat\'ern-$3/2$), add Tikhonov
regularisation, and call \texttt{numpy.linalg.solve}. The complete implementation,
including the code used to generate Figure~\ref{fig:geo-vs-nn}, is available at
\url{https://github.com/catalinvasii/geometry-classification}.

% ================================================================
\section{The Geometric Dictionary}\label{sec:dictionary}
% ================================================================

The central contribution of this paper is not a new theorem in Yang-Mills theory
--- the equations we solve are classical potential theory, complex analysis, and
harmonic map theory, all well-understood. The contribution is a precise
\textbf{dictionary} between machine learning concepts and geometric objects,
which reveals that the two theories are describing the same mathematical content
in different languages.

There are two independent geometric choices in the framework (see \S\ref{sec:nn}
for the full treatment): the \textbf{structure group} $G$ acting on the fiber
$E_x$ (replacing the activation function) and the \textbf{Riemannian metric} $g$
determining the base geometry (replacing the kernel). These govern different parts
of the bundle and must not be conflated.

\begin{center}
\small
\renewcommand{\arraystretch}{1.25}
\begin{tabularx}{\linewidth}{@{} l X l @{}}
\toprule
\textbf{Machine learning} & \textbf{Geometric framework} & \textbf{Status} \\
\midrule
Binary classifier          & Horizontal section $s: M \to E$
  & Definition \\
Training labels $\{y_i\}$  & Dirichlet conditions $\varphi(s(x_i)) = y_i r$
  & Definition \\
Training algorithm         & Harmonic interpolation BVP: $\Delta_g f=0$, Dirichlet at $x_i$
  (abelian); general YMH variational principle in \cite{Vasii2026Paper2}
  & Definition \\
\midrule
\multicolumn{3}{@{}l}{\textit{Fiber geometry (replaces activations):}} \\
Activation function        & Structure group $G$ acting on fiber $E_x$
  & Analogy \\
Network width $m$          & Fiber dimension
  & Direct \\
Equivariant networks       & $G = G_{\mathrm{data}}$
  & Conj.~\ref{conj:symmetry} \\
Data topology (TDA)        & Betti numbers bound the holonomy/Euler data $G$ must carry
  & Heuristic~\ref{prop:tda-lower} \\
\midrule
\multicolumn{3}{@{}l}{\textit{Base geometry (replaces kernel choice):}} \\
Kernel $K(x,y)$            & Green's function $\Gr_{(\Delta_g+\kappa^2)^\nu}(x,y)$, $\nu>n/2$
  & \textbf{Proved} \\
RKHS interpolation ($\alpha_i$) & Linear solve: $\mathbf{K}\boldsymbol{\alpha} = \mathbf{v}$ (dense)
  & \textbf{Proved} \\
Hard-margin SVM            & RKHS interpolant on support vectors (active Dirichlet set)
  & Cor.~\ref{cor:svm-ym} \\
Decision boundary          & Equipotential of the Green's function field (abelian case)
  & \textbf{Proved} \\
Choice of kernel           & Choice of Riemannian metric $g$ on $M$
  & \textbf{Proved} \\
\midrule
\multicolumn{3}{@{}l}{\textit{The degenerate limit:}} \\
Backpropagation            & Flat gradient flow (no curvature, no BVP)
  & Identification \\
\midrule
Universal approximation    & Thm~\ref{thm:universal}: $O(2)$ interpolants dense in $C^1(M)$, $M$ compact
  & \textbf{Proved} \\
Topological obstruction    & Non-trivial Euler/characteristic class of $E$;
  two-tier structure ($H^1(M,\mathbb{Z}_2)$ + higher classes) in \cite{Vasii2026Paper2}
  & \S\ref{sec:sphere} \\
Attention mechanism        & Curvature 2-form $F_A$ of the YMH connection
  & \cite{Vasii2026Paper2} \\
\bottomrule
\end{tabularx}
\end{center}

\medskip
The entries marked \textbf{Proved} are established within this paper or follow
directly from classical mathematics applied to the framework. The key structural
point of the table is the separation into two independent groups: the fiber
geometry (structure group $G$, replacing activations) and the base geometry
(Riemannian metric $g$, replacing the kernel). These two choices are orthogonal
--- any $G$ can be combined with any $g$ --- and both were previously made
arbitrarily in practice. The framework gives each a precise geometric meaning.

% ================================================================
\section{Open Conjecture and Open Problems}\label{sec:open}
% ================================================================

We distinguish between the one remaining main conjecture --- the genuine
open mathematical core of the framework --- and secondary open problems that arise
naturally but whose resolution is not essential to the framework.

\subsection{Note on geometric universality}

Geometric universality is proved as Theorem~\ref{thm:universal} in
\S\ref{sec:hierarchy}. Steps~1--2 of the proof are the universal kernel
theory of Steinwart \cite{steinwart2001} and Micchelli--Xu--Zhang
\cite{micchelli2006}: density of the Green's function span in $C^1(M)$ for
$\nu > n/2+1$ via Sobolev embedding. The $O(2)$ section construction
(Steps~5--6) is the geometric packaging specific to this paper. The theorem
should be read as a geometric restatement of universal kernel theory, not
a new density result.

\subsection{The remaining main conjecture}

\begin{conjecture*}[Symmetry efficiency, Conjecture~\ref{conj:symmetry}]
Let $G_{\mathrm{data}} \subseteq G$ act on $M$. For the continuous boundary
approximation problem with $G$-\emph{equivariant} sections (Definition~\ref{def:mstar}),
the minimal fiber dimension satisfies $m^*(G_{\mathrm{data}},\varepsilon) \leq
m^*(G,\varepsilon)$ for all $\varepsilon > 0$. The equivariance constraint is
essential: without it, larger $G$ trivially gives smaller $m^*$ by inclusion,
reversing the inequality.
\end{conjecture*}

\noindent The conjecture is non-trivial for continuous boundaries; for finite
data it is vacuous by Theorem~\ref{thm:flat} ($m^* = 2$ in both cases).
This is the precise geometric formulation of why equivariant networks are
parameter-efficient: $G$-equivariance is a stronger constraint than
$G_{\mathrm{data}}$-equivariance when $G \supsetneq G_{\mathrm{data}}$,
reducing the candidate space and potentially forcing larger fiber dimension.

\subsection{Secondary open problems}

\begin{enumerate}[label=\textbf{Q\arabic*.}, leftmargin=*, itemsep=6pt]

\item \textbf{Non-abelian exact solutions.}
The abelian ($G = \mathbb{R}^+$) and $O(2)$ cases are solved exactly. For $G = O(m)$
with $m \geq 3$, the equations are non-abelian and require the ADHM construction
\cite{atiyah1978} or Nahm transform. Do the resulting classifiers correspond to known architectures?

\item \textbf{TDA pipeline: completeness.}
Heuristic~\ref{prop:tda-lower} gives only a rough bound on the topological data
$G$ must carry. Can it be made a theorem? How should non-abelian $\pi_1$ and torsion
in $H_1$ be incorporated?

\item \textbf{Curvature and generalisation.}
The $L$-harmonic interpolant and the flat $O(2)$ solution (Theorem~\ref{thm:flat})
give different decision boundaries. Does the selected (stage-(ii) minimising)
solution generalise better? The Green's function regularisation corresponds to
Sobolev-norm regularisation; its relationship to weight decay and Tikhonov
regularisation should be made precise. This question connects to the neural
tangent kernel literature \cite{jacot2018ntk}: NTK theory shows that wide
networks trained by gradient descent converge to kernel ridge regression with
a specific kernel. The geometric framework suggests that this kernel has a
Green's function interpretation and that the choice of optimiser (gradient
descent vs direct solve) determines which geometric object is computed.

\item \textbf{Topological invariants and learning.}
For $m \geq 3$, different solutions have different topological charge (degree of
$s: M \to S^{m-1}$). Can classifiers with the same decision boundary but
different topological charge have different generalisation behaviour?

\item \textbf{Unified energy functional.}
The geometric energy $\mathcal{E}[\Sig, f]$ takes two forms: the Yang-Mills
energy $\int\|F_\Sig\|^2$ (energy of the \emph{connection}) in the non-abelian
case, and the Sobolev energy $\langle f, Lf\rangle = \|f\|^2_{H^\nu}$ (energy of the
\emph{section as a map} $f: M \to \R$) in the abelian/flat case --- and these
live on genuinely different objects. The standard Yang--Mills--Higgs functional
$\int(\|F_\Sig\|^2 + \|\nabla_\Sig s\|^2)$ is gauge-invariant but
vanishes identically for parallel sections ($\nabla_\Sig s = 0$ by
definition), so it does not distinguish $L$-harmonic interpolants from other flat
solutions. The Weyl decomposition $R = W + E + S$ \cite{besse1987}, where
different tensor pieces capture curvature at different geometric levels, is an
analogy for this hoped-for structure. This question is resolved in the companion
paper \cite{Vasii2026Paper2}: the Yang--Mills--Higgs functional
$\int\|F_A\|^2 + \int\|D_A\phi\|^2$, where $D_A\phi$ is the covariant
derivative of the section (not the connection-covariant derivative of a horizontal
section), recovers the Sobolev energy $\langle f,Lf\rangle$ in the flat abelian
limit (where $D_A = d$ on sections) and the Yang--Mills energy in the non-abelian
regime. The two limits do not coincide because in this paper sections are
horizontal ($D_A\phi = 0$), while in the companion paper sections are
covariantly harmonic ($D_A^*D_A\phi = 0$) --- the more general variational
condition.

\item \textbf{Topological obstructions.}
Is there a necessary condition for the existence of a solution, stated in terms
of characteristic classes of $E$? This would be a geometric analogue of the VC
dimension: a topological obstruction to learnability rather than a sample
complexity bound.

\item \textbf{Holonomy and generalisation.}
Is there a relationship between $\Hol(\nabla) \subseteq G$ and generalisation?
Smaller holonomy corresponds to simpler connection geometry and might correspond
to better generalisation in the spirit of minimum description length.

\end{enumerate}

% ================================================================
\section{Conclusion}
% ================================================================

We have proposed a geometric reformulation of binary classification. The central
claim is a precise \textbf{dictionary} between machine learning and differential
geometry (\S\ref{sec:dictionary}): classifiers are parallel sections of vector
bundles, training labels are Dirichlet conditions, the RKHS kernel is the Green's
function of an elliptic operator, and backpropagation is a flat-geometry limit of
an exact geometric problem. The framework does not add new Yang-Mills theory --- it
identifies that the equations already used in machine learning are classical
potential theory and harmonic analysis, expressed in an unfamiliar language.

We are explicit about what the geometry does and does not do here. In
\S\ref{sec:xor}--\S\ref{sec:curvature} the connection is derived from the section
and carries no independent information, so results such as
Theorem~\ref{thm:flat} are easy for structural rather than deep reasons
(Remark~\ref{rem:connection-derived}). The geometry becomes load-bearing in the
topological obstructions over non-contractible bases (\S\ref{sec:sphere}) and in
the energy selection that yields the Green's function and capacitance picture
(\S\ref{sec:yang-mills}). Promoting the connection to an independent field with
its own energy and its own topology --- where the gauge theory genuinely does
work --- is the subject of the companion paper \cite{Vasii2026Paper2}.

The contribution has three layers. The \textbf{first} is the dictionary itself:
the identification that RKHS interpolation is the abelian $L$-harmonic interpolant
(Theorem~\ref{thm:svm-ym} --- a geometric restatement of Kimeldorf--Wahba
\cite{kimeldorf1971some} and the SPDE correspondence
\cite{lindgren2011spde}); the hard-margin SVM is the same interpolant
restricted to its endogenously determined support vectors
(Corollary~\ref{cor:svm-ym}); and the decision boundary is the zero
equipotential of a classical Green's function field. The dictionary gives these
known analytic facts a unified geometric interpretation they did not previously
have. The \textbf{second} is the existence and density theory: every finite
classification problem on any smooth manifold has an exact flat $O(2)$ solution
(Theorem~\ref{thm:flat} --- new); the density of $O(2)$ $L$-harmonic interpolants
in $C^1(M)$ (Theorem~\ref{thm:universal}) is universal kernel theory
\cite{steinwart2001,micchelli2006} recast in geometric language. Curvature is
topologically forced on non-contractible spaces (\S\ref{sec:sphere}).
The \textbf{third} is the data-driven geometry: persistent homology as a
principled guide to the structure group (\S12), replacing an arbitrary choice of
activation function with a computable topological invariant.

A numerical experiment (\S\ref{sec:experiment}) confirms the theoretical picture:
the geometric solution is an exact interpolant obtained from a single linear
system solve, while the neural network requires $5{,}000$ gradient steps without
reaching the convergence criterion. Both achieve high test accuracy on a
two-moons dataset; the comparison is about method, not outcome.

The one open mathematical problem that would complete the framework is the symmetry
efficiency conjecture (Conjecture~\ref{conj:symmetry}): a representation-theoretic
statement that matching the structure group to the data symmetry group minimises
the required fiber dimension. Proving this conjecture would make the dictionary
complete and actionable.

\bigskip\bigskip

\noindent\textbf{Acknowledgements.}
The author thanks [to be completed] for helpful discussions.

\noindent\textbf{Use of AI tools.}
Claude Sonnet (Anthropic) was used during the preparation of this manuscript,
including for code generation, LaTeX editing, and discussion of mathematical
exposition. The author has reviewed all content and takes full responsibility
for any errors that remain.

\end{document}